\documentclass[11pt,reqno]{amsart}
\usepackage{amsbsy,amsfonts,amsmath,amssymb,amscd,amsthm}
\usepackage{mathrsfs,float}
\usepackage[mathcal]{euscript}
\usepackage{enumitem,graphicx,epstopdf}
\usepackage[abs]{overpic}
\usepackage[a4paper,text={6.1in,8in},centering]{geometry}
\usepackage{pxfonts,epstopdf}
\usepackage[colorlinks=true,linkcolor=blue,citecolor=green]{hyperref}
\usepackage{srcltx}
\usepackage[margin=20pt,font=small,labelfont=bf,textfont=it]{caption}
\usepackage{graphicx}
\usepackage{subcaption}
\usepackage{inputenc}
\usepackage{xfrac}

\usepackage{tikz}
\usepackage{pgfplots}
\pgfplotsset{compat=1.17}

\hypersetup{linktocpage}

\small\normalsize
\theoremstyle{plain}
\newtheorem{mainthm}{Theorem}

\newtheorem{maincor}{Corollary}
\newtheorem{teo}{Theorem}[section]
\newtheorem{lem}[teo]{Lemma}
\newtheorem{prop}[teo]{Proposition}
\newtheorem{cor}[teo]{Corollary}
\newtheorem{defi}[teo]{Definition}
\newtheorem{afir}{Claim}[teo]

\theoremstyle{definition}
\newtheorem{exem}[teo]{Example}

\theoremstyle{remark}
\newtheorem{obs}[teo]{Remark}

\makeatletter
\newcommand{\mylabel}[2]{#2\def\@currentlabel{#2}\label{#1}}
\makeatother

\DeclareMathOperator{\dH}{\mathrm{d_{H}}}
\DeclareMathOperator{\D}{\mathrm{d}}
\DeclareMathOperator{\OF}{\mathcal{O}^+_{\mathscr{F}}}
\DeclareMathOperator{\OG}{\mathcal{O}^+_{\mathscr{G}}}
\DeclareMathOperator{\Lip}{\mathrm{Diff}^0}
\DeclareMathOperator{\IFS}{\mathrm{IFS}}
\DeclareMathOperator{\bD}{\bar{\mathrm{d}}_0}

\newcommand{\eqdef}{\stackrel{\scriptscriptstyle\rm def}{=}}

\setlist[enumerate,1]{label=(\arabic*)}
\setlist[enumerate,2]{label=(\alph*)}

\begin{document}
\title[Strong foliations and IFS]{Minimal Strong Foliations in Skew-products of Iterated Function Systems
}
\begin{abstract}
We study locally constant skew-product maps over full shifts of finite symbols with arbitrary compact metric spaces as fiber spaces. We introduce a new criterion to determine the density of leaves of the strong unstable (and strong stable) foliation, that is, for its minimality. When the fiber space is a circle, we show that both strong foliations are minimal for an open and dense set of robust transitive skew-products. We provide examples where either one foliation is minimal or neither is minimal. Our approach involves investigating the dynamics of the associated iterated function system (IFS). We establish the asymptotic stability of the phase space of the IFS when it is a strict attractor of the system. We also show that any transitive IFS consisting of circle diffeomorphisms that preserve orientation can be approximated by a robust forward and backward minimal, expanding, and ergodic (with respect to Lebesgue) IFS. Lastly, we provide examples of smooth robust transitive IFSs where either the forward or the backward minimal fails, or both.
\end{abstract}
\author[Barrientos]{Pablo G.~Barrientos}
\address{\centerline{Instituto de Matem\'atica e Estat\'istica, UFF}
    \centerline{Rua M\'ario Santos Braga s/n - Campus Valonguinhos, Niter\'oi,  Brazil}}
\email{pgbarrientos@id.uff.br}
\author[Cisneros]{Joel Angel Cisneros}
\address{\centerline{Instituto de Matem\'atica e Estat\'istica, UFF}
    \centerline{Rua M\'ario Santos Braga s/n - Campus Valonguinhos, Niter\'oi,  Brazil}}
\email{joelangel@id.uff.br}

\maketitle
\thispagestyle{empty}

\section{Introduction}
Strong partially hyperbolic dynamics represent a large class of systems that exhibit a rich variety of dynamical behavior. Such systems are characterized by an invariant set having a partially hyperbolic splitting in three non-trivial bundles $E^s \oplus E^c \oplus E^u$. The extremal subbundles $E^s$ and $E^u$ are, respectively, uniformly contracting and uniformly expanding by the action of the derivative, and $E^c$ has an intermediate behavior. It is well known~\cite{HPS77} that there are two invariant foliations that are tangent, respectively, to $E^s$ and $E^u$. We refer to these foliations as the \emph{strong stable} and \emph{strong unstable foliation}, respectively.
These invariant foliations provide insights into the global behavior of the system, such as transitivity and ergodic properties.
\emph{Transitivity} is defined as the existence of a dense orbit, and when this property persists under small perturbations, it is referred to as \emph{robust transitivity}.

Recall that a
foliation is \emph{minimal} if every leaf is dense in the whole space and \emph{robustly minimal} when this property holds for
any small perturbation of the initial system.
Bonatti, Díaz, and Ures~\cite{BDU02} proved that for three-dimensional robustly transitive, strongly partially hyperbolic and non-hyperbolic diffeomorphisms either the strong stable or the strong unstable foliation is robustly minimal.
This result was later extended in~\cite{RHU07} to diffeomorphisms on manifolds of higher dimensions and in~\cite{Nob:15} to the context of robustly transitive attractors that are robustly non-hiperbolic and partially hyperbolic with one-dimensional center bundle. Nobili questioned whether these results could be extended to the broader context of robustly transitive proper sets, rather than just attractors. For instance, to transitive hyperbolic sets of a
\emph{skew product} $T(\omega,x)=(\tau(\omega),f(\omega,x))$ where the \emph{base} map $\tau$ is a \emph{horseshoe} (a full shift in finite symbols) and the \emph{fiber} maps $f(\omega,\cdot)$ are diffeomorphisms of a manifold, which  is called the \emph{fiber space}.
However, it should be noted that this generalization is not feasible for most transitive proper sets, as demonstrated by an example in~\cite{BCGP14} (as notified in~\cite{Nob:15}).

In~\cite{BDU02}, the robust minimality of both strong foliations is obtained by assuming a dynamically coherent central foliation, orientability of the bundles $E^s$, $E^c$, and $E^u$, preservation of this orientability by the differential map, and the existence of a periodic compact center leaf. A similar result was obtained for robust transitive strong partially hyperbolic attractors in~\cite[Proposition 9.1]{Nob:15}. In particular, examples having both strong foliations minimal are provided by robust transitive non-hyperbolic skew-product maps whose base dynamics is a hyperbolic transitive attractor and fiber space is $\mathbb{S}^1$, such as the one obtained in~\cite{BD96}. In view of this result, Nobili raises the question once again of whether there exists a strong partially hyperbolic attractor with a one-dimensional center bundle where only one foliation is minimal.

On the other hand, Pujals and Sambarino~\cite{PS06} gave conditions that ensure the robust minimality of strong stable foliation for partially hyperbolic diffeomorphisms defined on a compact manifold. They achieve this by introducing the \textrm{SH} property. With this technique, they obtained that one of the strong foliations of the classical examples of Shub~\cite{shub1971topological} and Mañé~\cite{mane1978contributions} (a certain class of Derived from Anosov) is robustly minimal.
Recently, in~\cite{RUY22}, the authors provide conditions that guarantee the robust minimality of the strong stable foliation for a general class of \emph{derived from Anosov} (isotopic to a linear Anosov) diffeomorphisms on the three torus. Furthermore, they construct an example that exhibits the robust minimality of both strong foliations.

In this work, we investigate a class of partially hyperbolic systems and their associated strong invariant foliations. We will examine the setting proposed by Nobili in the first aforementioned question of minimality for both strong foliations of robust transitive strong partially hyperbolic proper sets (other than attractors). To do this, we will model these sets as skew-products over a full shift of finite symbols, where the fiber space is an arbitrary compact metric space of any dimension. As far as we know, this is the first paper that provides a criterion for the minimality of the strong foliations when the central dimension is larger than one. However, we will obtain more specific results in the case of dimension one, i.e., when the fiber space is the circle. To simplify the analysis, we will consider locally constant skew-products as a toy model for this dynamics. This will allow us to reduce the analysis to an associated \emph{iterated function system} (IFS). An IFS can be thought of as a finite collection of continuous functions that can be applied successively in any order. We will also provide examples of maps where only one of the foliations is minimal or neither is minimal, addressing the second mentioned question of Nobili in our setting. On the other hand, it should be noted that in proving our theorems on skew-products, we have obtained significant and potentially independently interesting results in the field of IFS.

\subsection{Minimality of strong foliations} \label{sec:minimality-foliations-intro}
We introduce the notion of a \emph{symbolic one-step skew-product map}, also known in the literature as a \emph{locally constant skew-shift}.  First,  let $\Sigma_k=\{1,\dots,k\}^\mathbb{Z}$ be the space of bi-sequences of $k\geq 1$ symbols endowed with the distance
$$
  \D_{\Sigma_k}(\omega,\omega')=\nu^m, \quad m=\min\{i\geq 0: \omega_i \not = \omega'_i \mbox{ or } \omega_{-i} \not = \omega'_{-i}  \}
$$
where $\omega=(\omega_{i})_{i\in\mathbb{Z}},\omega'=(\omega'_{i})_{i\in\mathbb{Z}}\in \Sigma_k$ and $0<\nu<1$. We assume that $(X,\mathrm{d})$ is a metric space and consider a finite family $\mathscr{F}=\{f_{1},\dots,f_{k}\}$ of continuous maps on $X$.
Associated with  $\mathscr{F}$, we define the \emph{symbolic one-step skew-product map} $\Phi$
on the product space $\Sigma_k\times X$ by
\begin{equation}\label{iv.1}
 \Phi:\Sigma_{k}\times X\to\Sigma_{k}\times X,\quad\Phi(\omega,x)
    =(\tau(\omega),f_{\omega_{0}}(x))
\end{equation}
where $\omega=(\omega_{i})_{i\in\mathbb{Z}}\in\Sigma_{k}$ and $\tau:\Sigma_{k}\to\Sigma_{k}$ is the lateral shift map.
To emphasize the role of the fiber maps and the product structure, we will write $\Phi=\tau\ltimes \mathscr{F}$.

Let us denote by $W_{loc}^{s}(\omega)$ and $W_{loc}^{u}(\omega)$  the \emph{local stable} and \emph{local unstable} set for the shift map $\tau$
of a bi-sequence $\omega=(\omega_{i})_{i\in\mathbb{Z}}$. That is, the set of points $\omega'=(\omega'_i)_{i\in\mathbb{Z}} \in \Sigma_k$ where $\omega'_i=\omega_i$ for all
$i\geq 0$ and $\omega'_i=\omega_i$ for all $i<0$ respectively.  Now, we define the \emph{local strong stable} and the  \emph{local strong unstable} set of $(\omega,x)\in\Sigma_{k}\times X$ as
\begin{equation*}
W_{loc}^{ss}(\omega,x)=W_{loc}^{s}(\omega)\times\{x\}
   \quad\text{and} \quad
   W_{loc}^{uu}(\omega,x)=W_{loc}^{u}(\omega)\times\{x\}.
\end{equation*}
Finally, the \emph{global strong stable} and \emph{global strong unstable} leaves of $(\omega,x)\in\Sigma_{k}\times X$ are defined as follows:\begin{equation} \label{eq:ss-uu-global}
    \begin{aligned}
       W^{ss}(\omega,x )&=
       \bigcup_{n\geq0}{\Phi^{-n}\left(W_{loc}^{ss}(\Phi^{n}(\omega,x) )\right)} \quad  \text{and} \quad
       W^{uu}(\omega,x )=
       \bigcup_{n\geq0}{\Phi^{n}\left(W_{loc}^{uu}(\Phi^{-n}(\omega,x) )\right)}.
    \end{aligned}
\end{equation}
We construct a pair of invariant foliations of $\Sigma_k\times X$ using stable and unstable leaves, respectively. Specifically, we define the \emph{stable strong foliation} and \emph{unstable strong foliation} by
$$
\mathcal{F}^{ss}(\Phi)=\{W^{ss}(\omega,x ): (\omega,x)\in \Sigma_k\times X\} \quad \text{and} \quad \mathcal{F}^{uu}(\Phi)=\{W^{uu}(\omega,x ): (\omega,x)\in \Sigma_k\times X\}.
$$
A \emph{foliation}  is said to be \emph{minimal}   if every leaf is dense in the whole space.

A family $\mathscr{F}$ is said to be \emph{minimal} (resp.~\emph{transitive}) if every orbit of any point (resp.~any open set) under the action of the semigroup $\langle \mathscr{F} \rangle^+$ is dense in $X$.
In~\cite[Lemma~5.5]{HN13} the authors assert, without proof, that  the strong unstable foliation $\mathcal{F}^{uu}(\Phi)$ is minimal  under the assumption that $\mathscr{F}$ is minimal.  However, we will show that this claim is not true in general, as demonstrated by Examples~\ref{example_rotation} and Corollary~\ref{exemplo-pablo}.  In the following result, we establish a new criterion for the density of the leaves of the strong unstable foliation. We assume that $X$ is a \emph{strict attractor} of $\mathscr{F}$, which means that any compact set $S$ in $X$ converges in the Hausdorff distance to $X$ under the action of the Hutchinson operator defined as $F(S)=f_1(S)\cup \dots \cup f_k(S)$. For a more general concept of a strict attractor, see Definition~\ref{defi-strict-atractor}.

\begin{mainthm}\label{mainthm1}
Let $\mathscr{F}$ be a finite set of homeomorphisms of a compact metric space  $X$ and consider the associated symbolic one-step skew-product map $\Phi=\tau\ltimes \mathscr{F}$.
Suppose that $X$ is a strict attractor of~$\mathscr{F}$.  Then, the strong unstable foliation $\mathcal{F}^{uu}(\Phi)$ is minimal.
\end{mainthm}


As mentioned, the minimality of $\mathscr{F}$ does not imply the minimality of $\mathscr{F}^{uu}(\Phi)$. However,  under certain conditions, such as those given in~\cite[Prop.~3.9 and~3.10]{BFMS16}  a minimal IFS admits a strict attractor. Also, Sarizadeh communicated to us that a minimal IFS with an attracting periodic point admits a strict attractor; see~\cite{sarizadeh2023attractor} and Proposition~\ref{Sarizadeh} for details.  As a consequence, we obtain the following  sufficient condition to ensure the minimality of the foliation from the minimality of the IFS.

\begin{maincor} \label{maincor1}
Let $\Phi=\tau\ltimes \mathscr{F}$ be as in Theorem~\ref{mainthm1}.
Assume that $\mathscr{F}$ is minimal and there is $g\in \mathscr{F}$
with an attracting fixed point. Then, the strong unstable foliation $\mathcal{F}^{uu}(\Phi)$  is minimal.
\end{maincor}

Dual statements of Theorem~\ref{mainthm1} and Corollary~\ref{maincor1} holds for $\mathscr{F}^{ss}(\Phi)$ under the same  assumption  for the family $\mathscr{F}^{-1}=\{f^{-1}: f\in \mathscr{F}\}$ instead of $\mathscr{F}$.

\subsection{Genericity of minimality} \label{sec:gen-min-intro}
Fixed  $k\geq 1$ and $1\leq r\leq \infty$. Let $\mathrm{Diff}_+^r(\mathbb{S}^1)$ be the set of  orientation-preserving circle $C^r$ diffeomorphisms.
 We denote by $\IFS_k^r(\mathbb{S}^1)$ the set of families $\mathscr{F}=\{f_1,\dots,f_k\}$ with $f_i\in \mathrm{Diff}^r_+(\mathbb{S}^1)$ for every $i=1,\dots,k$, where we consider families equal, regardless of the order of their elements.  This identification of families is necessary to endow the set of iterated function systems with a metric. Namely, we consider here the metric in $\IFS_k^r(\mathbb{S}^1)$ given by
$$
  \bar\D_r(\mathscr{F},\mathscr{G})\eqdef \min_{\sigma \in S_k} \sum_{i=1}^k \D_r(f_i,g_{\sigma(i)}), \qquad \mathscr{F}=\{f_1,\dots,f_k\}, \ \ \mathscr{G}=\{g_1,\dots,g_k\} 
$$
where $S_k$ is the permutation group on $\{1,\dots,k\}$ and $\D_r$ is the $C^r$ distance in $\mathrm{Diff}^r_+(\mathbb{S}^1)$.

Define $\mathrm{S}_{k}^{r}(\mathbb{S}^{1})$ as the set of symbolic one-step skew-product  maps
$\Phi=\tau\ltimes \mathscr{F}$, where $\mathscr{F}\in \IFS_k^r(\mathbb{S}^1)$.  We endow  $\mathrm{S}_{k}^{r}(\mathbb{S}^{1})$ with a topology, saying that $\Phi=\tau\ltimes \mathscr{F}$
is close to $\Psi=\tau\ltimes \mathscr{G}$ if $\mathscr{F}$ and $\mathscr{G}$ are close in $\IFS_k^r(\mathbb{S}^1)$.
Denote by $\mathrm{SRT}_k^r(\mathbb{S}^1)$ the set of the symbolic one-step skew-product $\Phi$ which are robustly transitive, i.e. there is a neighborhood $\mathcal{U}$ in $\mathrm{S}_k^r(\mathbb{S}^1)$ of $\Phi$ such that any $\Psi$ in $\mathcal{U}$ is transitive (that is, there is a dense orbit).

\begin{mainthm} \label{mainthm2}
%
%
%
For $k\geq 2$ and $1\leq r\leq \infty$, there is a dense and open subset $\mathcal{R}$ of $\mathrm{SRT}_{k}^{r}(\mathbb{S}^{1})$ such that both the strong unstable foliation $\mathcal{F}^{uu}(\Phi)$ and strong stable foliation $\mathcal{F}^{ss}(\Phi)$ are minimal for all $\Phi\in\mathcal{R}$. Moreover, any transitive skew-product $\Psi \in \mathrm{S}^r_k(\mathbb{S}^1)$ can be approximated by maps in $\mathcal{R}$.
\end{mainthm}

The following result highlights the difference between the open and dense set established in the above theorem and the set of robustly transitive one-step symbolic skew-products.

\begin{mainthm} \label{mainthm-examples}
    For $k\geq 8$ and $1\leq r\leq \infty$, there are maps $\Phi \in \mathrm{SRT}^r_{k+1}(\mathbb{S}^1)$ and $\Psi \in \mathrm{SRT}^r_k(\mathbb{S}^1)$  such that
    \begin{itemize}
        \item $\mathcal{F}^{uu}(\Phi)$ is minimal, but $\mathcal{F}^{ss}(\Phi)$ is not minimal;
        \item neither $\mathcal{F}^{uu}(\Psi)$ nor $\mathcal{F}^{ss}(\Psi)$ is minimal.
    \end{itemize}
\end{mainthm}

\subsection{Iterated Function Systems} \label{IFS-intro}
To prove Theorem~\ref{mainthm1}, we require the asymptotic stability of a strict attractor. Rypka claimed to have established this property in~\cite[Theorem~3.1]{MR19}, but we encountered difficult to understand the proof, which we explain in detail in Remarks~\ref{rem:Rypka} and~\ref{rem:error}. However, we overcame these difficulties by providing a correct proof and generalizing the result to include more general notions of monotonicity assumptions in the case that the strict attractor is the entire space. Specifically, we consider any continuous map \mbox{$F:\mathscr{K}(X)\to\mathscr{K}(X)$}, where $\mathscr{K}(X)$ is the hyperspace of nonempty compact subsets of a metric space $X$ endowed with the Hausdorff metric $\dH$. We say that $F$ is \emph{monotone} if $F(A)\subset F(B)$ for any $A,B\in\mathscr{K}(X)$ with $A\subset B$. Compare with~\cite[Def.~2.7 and Rem.~2.9]{MR19}.

The compact metric space $X$ is said to be an \emph{attractor} of $F$ if there exists an open set $\mathcal{U}\subset\mathscr{K}(X)$ such that $X\in \mathcal{U}$ and $F^n(C)$ converges to $X$ in the Hausdorff distance for any $C\in \mathcal{U}$. Note that by the continuity of $F$, we have $F(X)=X$. Furthermore, $X$ is called \emph{asymptotically stable} if it is an attractor of $F$ and for any $\varepsilon>0$, there exists $\delta>0$ such that $\dH(F^{n}(S),X)<\varepsilon$ for all $n\geq 0$ and $S\in\mathscr{K}(X)$ with $\dH(S,X)<\delta$.

\begin{mainthm} \label{mainthm-stability}
Consider a compact metric space $X$ and let $F:\mathscr{K}(X)\to\mathscr{K}(X)$ be a continuous monotone map.
If $X$ is an attractor of $F$, then $X$ is  asymptotically stable.  In particular, if $X$ is a strict attractor of an IFS, it is
also asymptotically stable.
\end{mainthm}

To prove Theorem~\ref{mainthm2}, we need to study the properties of minimality and transitivity for IFSs. Let $\mathscr{F}$ be an element of $\IFS^r_k(\mathbb{S}^1)$. We say that $\mathscr{F}$ is \emph{robustly transitive} (resp.~\emph{minimal}) if there exists a neighborhood $\mathcal{U}$ of $\mathscr{F}$ in $\IFS^r_k(\mathbb{S}^1)$ such that any $\mathscr{G}\in \mathcal{U}$ is transitive (resp.~minimal). We define $\mathrm{RT}^r_k(\mathbb{S}^1)$ to be the subset of $\IFS_k^r(\mathbb{S}^1)$ such that $\mathscr{F}$ is robustly transitive. Note that $\mathrm{RT}^r_k(\mathbb{S}^1)$ is an open set in $\IFS^r_k(\mathbb{S}^1)$.
For simplicity, we sometimes refer to $\mathscr{F}$ as \emph{forward (resp.~backward) minimal/transitive} if $\mathscr{F}$ (resp.~$\mathscr{F}^{-1}$) is minimal/transitive.

The first step to prove~Theorem~\ref{mainthm2} is the following result.

\begin{mainthm}\label{mainthm-IFS}
For $k\geq 2$ and $1\leq r\leq \infty$, there is an open and dense set $\mathcal{W}$ in $\mathrm{RT}^r_k(\mathbb{S}^1)$ such that every $\mathscr{F}\in \mathcal{W}$ is, both,  forward and backward robustly minimal. Additionally, $\mathbb{S}^1$ is a strict attractor of both~$\mathscr{F}$ and~$\mathscr{F}^{-1}$.

Furthermore, any transitive family $\mathscr{G}$ in $\IFS^r_k(\mathbb{S}^1)$  can be approximated by IFSs in $\mathcal{W}$.
\end{mainthm}

We define $\mathscr{F}\in \IFS_k^r(\mathbb{S}^1)$ as \emph{ergodic} (with respect to the normalized Lebesgue measure $\mathrm{Leb}$ on $\mathbb{S}^1$) if $\mathrm{Leb}(A)\in\{0,1\}$ for all Borel set $A$ in $\mathbb{S}^1$ such that $f(A)\subset A$ for all $f\in\mathscr{F}$. It is a well-known conjecture in the context of finitely generated smooth groups on the circle that every smooth minimal action is ergodic with respect to the Lebesgue measure. This conjecture has been proven to be true in many cases using the exponential expansion strategy (see references in~\cite{DKN09,DKN}). The question was also addressed for semigroup actions in~\cite{BFMS16} and proved for expanding IFSs comprised of $C^r$ circle diffeomorphisms with $r>1$ (see also~\cite{BF20}). Recall that $\mathscr{F}$ is said to be \emph{expanding} if for every $x\in \mathbb{S}^1$, there exists $g\in \langle \mathscr{F} \rangle^+$ such that the derivative $Dg^{-1}(x)>1$. As a consequence of this result and Theorem~\ref{mainthm-IFS}, we will obtain in Section~\S\ref{sec:minimal-in-rt-S1} the following.

\begin{maincor} \label{cor:ergodicity}
    Let $\mathcal{W}\subset \mathrm{RT}^r_k(\mathbb{S}^1)$ be as in Theorem~\ref{mainthm-IFS}. For $k\geq 2$ and $1< r\leq \infty$, arbitrarily close to any transitive family $\mathscr{F}$ in $\IFS_k^r(\mathbb{S}^1)$, there is a family $\mathscr{G} \in \mathcal{W}$ and a neighborhood $\mathcal{U}$ of $\mathscr{G}$ such that every $\mathscr{H}\in \mathcal{U}$ is ergodic.
\end{maincor}

Examples of IFSs comprising circle homeomorphisms that are forward minimal but not backward minimal were provided in~\cite{BGMS17} and~\cite{kleptsyn2018classification}. These examples are not necessarily smooth and not necessarily robustly transitive. In the following theorem, we provide new examples of IFSs comprising $C^r$ diffeomorphisms with $r\geq 1$ that are forward minimal but not backward minimal, and are also robustly transitive. Additionally, we present examples of robustly transitive IFSs that are neither forward nor backward minimal. These examples highlight the distinction between the sets $\mathcal{W}$ and $\mathrm{RT}^r_k(\mathbb{S}^1)$ in Theorem~\ref{mainthm-IFS} and support the examples described in Theorem~\ref{mainthm-examples}.

\begin{mainthm} \label{mainteo:exemplos-IFS} For any $r\geq 1$, $k\geq 8$, there exist $\mathscr{F} \in \mathrm{RT}^r_{k+1}(\mathbb{S}^1)$ and  $\mathscr{G} \in \mathrm{RT}^r_k(\mathbb{S}^1)$ such that
\begin{itemize}
    \item $\mathscr{F}$ is forward minimal but not backward minimal;
    \item $\mathscr{G}$ is not forward nor backward minimal.
\end{itemize}
\end{mainthm}

To prove Theorem~\ref{mainteo:exemplos-IFS}, we need to study the properties of (robust) minimality and (robust) transitivity of $\mathscr{F}$. Specifically, we need to extend the results in~\cite{BFMS16} to transitive IFSs. To accomplish this, we first revisit the theory of expanding minimal actions and blending regions for the more general context of IFSs of locally Lipschitz homeomorphisms in compact metric spaces in \S\ref{sec:criterion_minimality} and \S\ref{subsec: Blending-region}. In \S\ref{sec:rob-tran}, we then extend this theory to transitive IFSs.

\subsection{Organization of the paper}
The paper is organized as follows. In Section~\ref{sec:2}, we introduce some preliminaries of IFS and prove Theorem~\ref{mainthm-stability}. Sections~\ref{sec:minimality} and~\ref{sec:transitividad} are dedicated to the study of minimal and transitive IFSs. In Section~\ref{sec:minimal-in-rt-S1}, we will prove Theorems~\ref{mainthm-IFS} and~\ref{mainteo:exemplos-IFS}. Finally, in Section~\ref{sec:skew-prodcut}, we will prove Theorems~\ref{mainthm1},~\ref{mainthm2} and~\ref{mainthm-examples}.

\section{Asymptotic stability of attractors of IFSs}
\label{sec:2}
In what follows, $(X,\mathrm{d})$ denotes a metric space. Sometimes we will simply write $X$ if the metric is clear from context. We also use the notation $\mathscr{F}=\{f_{1},\dots,f_{k}\}$ to denote a finite family of $k\geq 1$ continuous functions $f_i$ from $X$ to itself, $i=1,\dots,k$, which we refer to as an \emph{iterated function system} (IFS). Note that strictly speaking, a finite family is an ordered list of elements, which may include repeated elements.


\subsection{Hyperspace of compact sets}

We define the \emph{hyperspace} $\mathscr{K}(X)$ as the set of all nonempty compact subsets of $X$. For $K,L\in \mathscr{K}(X)$, we define the \emph{Hausdorff distance}  as follows
\begin{equation*}
\mathrm{d_{H}}(K,L)\eqdef \max\{\max_{x\in K}{\mathrm{d}(x,L)}, \ \max_{y\in L}{\mathrm{d}(y,K)}\}.
\end{equation*}
It is easy to verify that $\mathrm{d_{H}}$ is a metric on $\mathscr{K}(X)$.

Given $A\in \mathscr{K}(X)$ and $r>0$, we use $B(A,r)$ and $B_r(A)$ to denote the open ball of radius $r$ centered at $A$ in $(\mathscr{K}(X),\mathrm{d_H})$ and in $(X,\mathrm{d})$, respectively. That is,
\begin{equation*}
B(A,r)\eqdef\{B\in\mathscr{K}(X):\mathrm{d_H}(A,B)<r\} \quad \text{and} \quad
B_{r}(A)=\{x\in X:  \mathrm{d}(x,A)<r\}.
\end{equation*}
Using this notation, we can provide an alternative definition of the Hausdorff metric by
\begin{equation*}
\mathrm{d_{H}}(K,L)\eqdef\inf\{r>0:K\subset B_{r}(L)\quad\text{and}\quad L\subset B_{r}(K)\}.
\end{equation*}
\subsection{Hutchinson operator.} \label{sec:Hutchinson}
Associated with the family $\mathscr{F}=\{f_1,\dots,f_k\}$ we define the \emph{Hutchinson} operator $F=F_{\mathscr{F}}$ on $\mathscr{K}(X)$ by
\begin{equation*}
F:\mathscr{K}(X)\to\mathscr{K}(X),\quad\quad F(A)=\bigcup_{i=1}^{k}{f_{i}(A)}.
\end{equation*}
The operator $F$ is a standard construction in the theory of IFSs; see, e.g.,~\cite{HUT81} for more details. The Hutchinson operator $F$ inherits some properties of the elements of the family~$\mathscr{F}$. For instance, since $\mathscr{F}$ is a finite family of continuous maps, the Hutchinson operator $F$ is also continuous, according~\cite[Thm~1]{barnsley2015continuity}.

\subsection{Attractors}
We consider a continuous map $F:\mathscr{K}(X)\to\mathscr{K}(X)$ and introduce the classical notion of an attracting fixed point for this particular case of functions.

\begin{defi} \label{def:attractor}
A compact set $K\in\mathscr{K}(X)$ is said to be an \emph{attractor} of $F$ if there exists an open set $\mathcal{U}\subset\mathscr{K}(X)$ such that $K\in \mathcal{U}$ and $F^{n}(C) \to K$ in the Hausdorff distance for any $C\in \mathcal{U}$.
\end{defi}

\begin{obs}
Any attractor $K$ is a fixed point of $F$.
Indeed, if $F^{n}(C)\to K$ as $n\to\infty$ in the Hausdorff metric, by the continuity of $F$ we have that $F^{n+1}(C)\to F(K)$. By the uniqueness of the limit, we conclude that $F(K)=K$.
\end{obs}

Following~\cite{barnsley2012real}, we introduce the notion of a strict attractor for the continuous map $F$.

\begin{defi}\label{defi-strict-atractor}
A compact set $K\in\mathscr{K}(X)$ is a \emph{strict attractor} of $F$ if there exists an open set $U\subset X$ such that $K\subset U$ and
$F^n(S)\to K$ in the Hausdorff distance for any $S\in\mathscr{K}(U)$.
\end{defi}

The notion of a strict attractor is stronger than the notion of an attractor, as we  see below.

\begin{prop}\label{propo-stric atractor-implica-atractor}
Every strict attractor is an attractor.
\end{prop}
\begin{proof}
Let $K\in\mathscr{K}(X)$ be a strict attractor of $F$. Then there exists an open set $U\subset X$ such that $K\subset U$ and $F^{n}(S)\to K$ in the Hausdorff distance for all $S\in\mathscr{K}(U)$. According to~\cite[Prop.~2.4.14]{srivastava2008course}, since $U$ is an open set of $X$, then $\mathscr{K}(U)$ is an open set in $\mathscr{K}(X)$. Therefore, $K$ is an attractor of $F$.
\end{proof}

Since the Hutchinson operator $F$ associated with a finite family $\mathscr{F}$ of continuous maps from $X$ to itself is a continuous map of the hyperspace of compact sets, the notions of an attractor and a strict attractor can be applied to $F$. For simplicity, we will refer to them as the \emph{attractor} or \emph{strict attractor of $\mathscr{F}$}.

\subsection{Stability}\label{sec:estabilidad-atractor}
Let  $F:\mathscr{K}(X)\to\mathscr{K}(X)$ be a continuous map and recall the following.

\begin{defi}
A fixed point $K\in\mathscr{K}(X)$ of $F$ is called \emph{stable} if for any $\varepsilon>0$, there exists $\delta>0$ such that for all $n\geq 0$ and $S\in\mathscr{K}(X)$ with $\dH(S,K)<\delta$, we have $\dH(F^{n}(S),K)<\varepsilon$. An attractor $K\in\mathscr{K}(X)$ of $F$ that is stable in this sense is said to be \emph{asymptotically stable}.
\end{defi}

\begin{defi} \label{def:monotonic}
The map $F$ is called \emph{monotone} if $F(A)\subset F(B)$ for any $A,B\in\mathscr{K}(X)$ with $A\subset B$.
\end{defi}

\begin{obs} \label{rem:Rypka}  Under a different and stronger assumption of monotonicity (see~\cite[Def.~2.7 and Rem.~2.9]{MR19}), Rypka states in~\cite[Lemma 3.2]{MR19} that if the entire space $X$ is an attractor of a continuous monotone map of $\mathscr{H}(X)$, then it is stable. We attempted to prove this result using the arguments in~\cite{MR19}, but we encountered difficulty understanding the proof.\footnote{On~\cite[p.103]{MR19}, the set $\mathcal{B}_{n_1n_2}=\{B\in \mathcal{B}_{n_1}: \ \text{there exists} \ B'\in \mathcal{B}_{n_2} \ \text{such that} \ \ B \subset B' \}$ is defined from a sequence of open sets $\mathcal{B}_n$ of $\mathbb{K}(A^{*})$. The author claims that for any $n_1 \in \mathbb{N}$ there exists $n_2 \in \mathbb{N}$ such that $\mathcal{B}_{n_1n_2}$ is non-empty. However, we do not understand the argument provided to prove that claim. Compare with the argument used to prove Claim~\ref{afirmacion-lema-estric-attractor-implica-estabilidad}.} Below, we present a different proof from~\cite{MR19} using the monotonic assumption in Definition~\ref{def:monotonic}.
\end{obs}

\begin{proof}[Proof of Theorem~\ref{mainthm-stability}]
Since $X$ is an attractor of $F$, there exists $\delta>0$ such that
\begin{equation}\label{ii.3}
\lim_{n\to\infty}{\dH(F^{n}(S),X)=0},
\quad \text{for all}\quad S\in B(X,\delta).
\end{equation}
If $X$ is a singleton, then it is obviously stable. Suppose that $X$ is not a singleton. We will proceed by contradiction. Assume that $X$ is not stable. Hence, there exist $\varepsilon>0$ and sequences $(k_{n})_{n\in\mathbb{N}} \subset \mathbb{N}$ and $(S_{n})_{n\in\mathbb{N}}\subset\mathscr{K}(X)$ such that
\begin{equation}\label{ii.4}
\dH(S_{n},X)<\min \left\{\frac{1}{n},\frac{\delta}{2} \right\} \quad \text{and} \quad
\dH(F^{k_{n}}(S_{n}),X)>\varepsilon.
\end{equation}
Observe that $k_{n}\to\infty$. Otherwise, $F$ might not be continuous. Since $F$ is continuous, for any $n\in\mathbb{N}$, the set $\mathcal{B}_{n}\subset\mathscr{K}(X)$ defined by $$\mathcal{B}_{n}=\{S\in B(X,\delta/2): \dH(F^{k_{n}}(S),X)>\varepsilon\}$$ is open and non-empty, since $S_{n}\in\mathcal{B}_{n}$.

\begin{afir}\label{afirmacion-lema-estric-attractor-implica-estabilidad}
Fix $n\in\mathbb{N}$ and let $\mathcal{B}\subset\mathscr{K}(X)$ be a non-empty open set such that $\mathcal{B}\subset\mathcal{B}_{n}$.
Then there exists $m>n$ such that $\mathcal{B}\cap\mathcal{B}_{m}\neq\emptyset$.
\end{afir}

Before to prove the above claim, let us explain how to use it to conclude the proof of the theorem.
Given $n_{1}\in\mathbb{N}$, since $\mathcal{B}^{1}\eqdef\mathcal{B}_{n_{1}}$ is a  non-empty  open set,
by Claim~\ref{afirmacion-lema-estric-attractor-implica-estabilidad} there is $n_{2}>n_{1}$  such that $\mathcal{B}^1\cap\mathcal{B}_{n_{2}}\neq\emptyset$. Applying the claim again to the non-empty open set   $\mathcal{B}^{2}\eqdef\mathcal{B}^{1}\cap\mathcal{B}_{n_{2}}\subset\mathcal{B}_{n_{2}}$, we obtain $n_{3}>n_{2}$ such that $\mathcal{B}^{2}\cap \mathcal{B}_{n_{3}}\neq\emptyset$. Repeating this process inductively, we obtain a sequence
$(\mathcal{B}^n)_{n\geq 1}$ of non-empty open sets in $\mathscr{K}(X)$. Observe that this sequence is nested, i.e. $\mathcal{B}^{n+1}\subset \mathcal{B}^n$ for all $n\geq 1$. Moreover, for any $i\in\mathbb{N}$, we also have
\begin{equation*}
    \dH(F^{k_{n_{i}}}(S),X)>\varepsilon \quad\text{with}\quad S\in\mathcal{B}^{i}\subset\mathcal{B}_{n_{i}}.
\end{equation*}
Let $\overline{\mathcal{B}^n}$ be the closure (in Hausdorff topology) of $\mathcal{B}^n$. Since $X$ is compact, we have that $\mathscr{K}(X)$ is also compact (c.f.~\cite[Thm.~2.4.17]{srivastava2008course}). Thus, $\overline{\mathcal{B}^{n}}$ is a compact subset of $\mathscr{K}(X)$ for any $n$.
Since the sequence $(\overline{\mathcal B^n})_{n\geq 1}$ is also nested, the intersection
$\overline{\mathcal{B}^{\infty}}\eqdef\bigcap_{n=1}^{\infty}{\overline{\mathcal{B}^{n}}}$ is not empty.
As $F$ is continuous, we have
\begin{equation*}
    \dH(F^{k_{n_{i}}}(S),X)\geq\varepsilon,\quad i\in\mathbb{N},\quad S\in\overline{\mathcal{B}^{\infty}}\subset B(X,\delta).
\end{equation*}
This contradicts Equation~\eqref{ii.3}. Therefore, $X$ is stable (and then asymptotically stable since it is also an attractor).

Moreover, as the Hutchinson operator is a continuous and monotonic function defined on the hyperspace of compact sets, we can apply the first part of this theorem when $X$ is a strict attractor of an  IFS. This completes the proof of Theorem~\ref{mainthm-stability}.
\end{proof}
\begin{proof}[Proof of Claim~\ref{afirmacion-lema-estric-attractor-implica-estabilidad}]
Since $\mathcal{B}$ is a non-empty open set, there are $C\in\mathcal{B}$ and $r>0$ such that $B(C,r)\subset\mathcal{B}$.
Consider $B_{\sfrac{r}{4}}(C)=\{x\in X:~\mathrm{d}(x,C)<r/4\}$. The closure $\overline{B_{\sfrac{r}{4}}(C)}$ (with respect to the topology induced by $\mathrm{d}$) is a closed subset of $X$, and since $X$ is compact, $\overline{B_{\sfrac{r}{4}}(C)}$ is itself a compact subset of $X$.

Observe that $\overline{B_{\sfrac{r}{4}}(C)}$ is a subset of the union of balls $B(x,r/4)=\{y\in X:~\mathrm{d}(x,y)<r/4\}$, where $x$ is in $\overline{B_{\sfrac{r}{4}}(C)}$.
By compactness, we can choose  $x_1,\dots,x_p \in \overline{B_{\sfrac{r}{4}}(C)}$ such that
\begin{equation}\label{ii.5}
  \overline{B_{\sfrac{r}{4}}(C)}\subset\bigcup_{i=1}^{p}{B(x_i,r/4)}.
\end{equation}
For any $z\in {B(x_1,r/4)}\cup\dots \cup {B(x_p,r/4)} $, there is $j\in\{1,\dots,p\}$ such that $\mathrm{d}(z,x_{j})<r/4$.  Moreover, since $x_{j}\in\overline{B_{\sfrac{r}{4}}(C)}$, we have $\mathrm{d}(x_{j},C)\leq r /4$, which implies that $\mathrm{d}(z,C)<r/2$. Therefore,
\begin{equation}\label{ii.6}
    \bigcup_{i=1}^{p}{B(x_i,r/4)}\subset B_{\sfrac{r}{2}}(C).
\end{equation}
On the other hand, from~\eqref{ii.4}, we have that $\dH(S_{m},X)\to 0$ as $m\to\infty$. Then there is $m>n$ such that
$S_{m}\cap B(x_i,r/4)\neq\emptyset$ for all $i=1,\dots,p$.
Take $S'_{m}=  S_{m} \cap \overline{B_{\sfrac{r}{2}}(C)} \in \mathscr{K}(X)$. From~\eqref{ii.6}, we have
\begin{enumerate}[label=(\roman*)]
    \item\label{item1-afirmacion-lema-estric-attractor-implica-estabilidad} $S'_{m}\cap B(x_i,r/4)\neq\emptyset~\text{for all}~i=1,\dots,p$,
    \item\label{item2-afirmacion-lema-estric-attractor-implica-estabilidad}$S'_{m}\subset B_{r}(C)$.
\end{enumerate}
We can prove two additional statements:
\begin{enumerate}[label=(\alph*)]
\item \label{item-a-afirmacion-lema-estric-attractor-implica-estabilidad} $\dH(S'_{m},C)<r$,
\item \label{item-b-afirmacion-lema-estric-attractor-implica-estabilidad} $\dH(F^{k_{m}}(S'_{m}),X)>\varepsilon$.
\end{enumerate}

To see~\ref{item-a-afirmacion-lema-estric-attractor-implica-estabilidad}, we use~\eqref{ii.5} to obtain $C\subset \overline{B_{\sfrac{r}{4}}(C)}\subset B(x_1,r/4)\cup \dots \cup B(x_p,r/4)$.
For any $c \in C$, there exists $j \in {1,\dots,p}$ such that $c \in B(x_j,r/4)$. By~\ref{item1-afirmacion-lema-estric-attractor-implica-estabilidad}, we can find $x' \in B(x_j,r/4) \cap S'_{m}$, and then by the triangle inequality, we get $\mathrm{d}(c,x')<r/2$. Thus, $c \in B_{\sfrac{r}{2}}(S'_{m})$, which implies $C \subset B_{\sfrac{r}{2}}(S'_{m})$. By combining this with \ref{item2-afirmacion-lema-estric-attractor-implica-estabilidad}, we obtain \ref{item-a-afirmacion-lema-estric-attractor-implica-estabilidad}.

To prove~\ref{item-b-afirmacion-lema-estric-attractor-implica-estabilidad}, we note that since $S_{m}\in\mathcal{B}_{m}$, we have $\dH(F^{k_{m}}(S_{m}),X)>\varepsilon$. Since $S'_{m}\subset S_{m}$ and $F$ is monotone, we have
$\dH(F^{k_{m}}(S'_{m}),X)\geq \dH(F^{k_{m}}(S_{m}),X)>\varepsilon$.

To conclude the proof of the claim, we use~\ref{item-a-afirmacion-lema-estric-attractor-implica-estabilidad} to obtain
\begin{equation}\label{ii.7}
    S'_{m}\in B(C,r) \subset\mathcal{B}\subset\mathcal{B}_{n}=\{S\in B(X,\delta/2):\dH(F^{k_{n}}(S),X)>\varepsilon\}.
\end{equation}
In particular,  $S'_m \in B(X,\delta/2)$, and using~\ref{item-b-afirmacion-lema-estric-attractor-implica-estabilidad}, we conclude that $S'_m \in \mathcal{B}_m$. Therefore, combining this with~\eqref{ii.7}, we have $S'_m \in \mathcal{B} \cap \mathcal{B}_m$.
\end{proof}

\begin{obs} \label{rem:error} In~\cite[Theorem 3.1]{MR19}, Rypka used his version of Theorem~\ref{mainthm-stability} to extend the stability result to general attractors of $F$. However, once again, we find that the proof is incorrect. Besides numerous misspellings, the proof provided by Rypka does not result in a contradiction as claimed.\footnote{In the notation of the proof of \cite[Theorem 3.1]{MR19}, it is necessary to prove $\dH(B^{i_n-k_0}_n,C^{i_n-k_0}_n)<\delta_0$ to conclude the contradiction. This is not done in the proof. In fact, the contradiction fails in the case where $F$ is the Hutchinson operator of a function $f$ and $A^*$ is a singleton.}  Therefore, it remains unknown whether the stability of general attractors is true or not.
\end{obs}

The following consequence of Theorem~\ref{mainthm-stability} will enable us to prove the minimality of the strong unstable foliation in Theorem~\ref{mainthm1}. First, recall that  given $\varepsilon>0$, a subset $D$ of a metric space $(X,\mathrm{d})$ is $\varepsilon$-\emph{dense} if for every $x\in X$, there is a $y\in D$
such that $\mathrm{d}(x,y)<\varepsilon$.

\begin{prop}\label{propo-previo-minimalidad}
Consider a compact metric space $X$ and let $F:\mathscr{K}(X)\to\mathscr{K}(X)$ be a continuous monotone map. Suppose that $X$ is a strict attractor of $F$. Then, for every $\varepsilon>0$, there exists $n_{0}\in\mathbb{N}$ such that for any $n\geq n_{0}$ and for all $y\in X$, we have $F^{n}({y})$ is $\varepsilon$-dense in $X$.
\end{prop}

\begin{proof}
Since $X$ is a strict attractor of $F$ by Proposition~\ref{propo-stric atractor-implica-atractor},
$X$ is an attractor of $F$. Thus, by Theorem~\ref{mainthm-stability}, we have that $X$ is stable.
Then given $\varepsilon>0$, there is $\delta>0$ such that for any $S\in\mathscr{K}(X)$
with $\dH(S,X)<\delta$, we have that
\begin{equation*}
\dH(F^{m}(S),X)<\frac{\varepsilon}{2} \quad  \text{for all $m\geq0$}.
\end{equation*}
On the other hand, since $X$ is a strict attractor of $F$ we have $\lim_{n\to\infty}{\dH(F^{n}(S),X)}=0$ for all
$S\in\mathscr{K}(X)$. In particular, given $y\in X$, we have
$    \lim_{n\to\infty}{\dH(F^{n}(\{y\}),X)}=0$.
From this, there exists $n_{y}\in\mathbb{N}$ such that $\dH(F^{n_{y}}(\{y\}),X)<\delta$. From the continuity of $F^{n_{y}}$, there is a neighborhood $B_{y}$ of $y$ such that
$\dH(F^{n_{y}}(\{x\}),X)<\delta$ for all $x\in B_{y}$. Thus, we can find a cover of $X$ by neighborhoods of the form $B_{y}$, that is,
\begin{equation*}
    X\subset \bigcup_{y\in X}{B_{y}}\quad\text{with}\quad\dH(F^{n_{y}}(\{x\}),X)<\delta\quad\text{for all}\ x\in B_{y}.
\end{equation*}
By the compactness of $X$, there are $n_{1},n_{2},\dots,n_{r}\in\mathbb{N}$ and open sets
$B_{1},B_{2},\dots,B_{r}$ such that
\begin{equation*}
    X\subset \bigcup_{i=1}^{r}{B_{i}}\quad\text{with}\quad\dH(F^{n_{i}}(\{x\}),X)<\delta\quad\text{for all}\ x\in B_{i}.
\end{equation*}
From here, by stability of $X$ we have
\begin{equation*}
  \dH(F^{m}(F^{n_{i}}(\{x\})),X)<\frac{\varepsilon}{2}
  \quad\text{for all} \  m\geq 0\ \text{and} \  x\in B_{i}, \ i=1,\dots,r.
\end{equation*}
Taking $n_{0}=\max\{n_{1},n_{2},\dots,n_{r}\}$, we have
\begin{equation}\label{ii.8}
  \mbox{d}_{H}(F^{n}(\{y\}),X)<\frac{\varepsilon}{2}\quad\text{for all}\ n\geq n_{0}\ \text{and}\ y\in X.
\end{equation}
From~\eqref{ii.8}, we have $X\subset B_{\frac{\varepsilon}{2}}(F^{n}(\{y\}))$ for all $n\geq n_{0}$ and $y\in X$.
On the other hand, note that $B_{\frac{\varepsilon}{2}}(F^{n}(\{y\}))\subset X$. So we have
    $X=B_{\frac{\varepsilon}{2}}(F^{n}(\{y\}))$.
This immediately implies that
$F^{n}(\{y\})$ is $\varepsilon$-dense on $X$ for all $n\geq n_{0}$ and  $y\in X$.
\end{proof}


\section{Minimality of IFSs} \label{sec:minimality}


Let $\mathscr{F}$ be a finite family  of continuous maps of $X$ to itself as before.
Let us denote by $\langle \mathscr{F} \rangle^{+}$ the semigroup generated by $\mathscr{F}$.  That is,
\begin{equation*}
\langle \mathscr{F} \rangle^{+} \eqdef \{g_n\circ\dots\circ g_1: \text{$g_i \in \mathscr{F}$ and $n\in\mathbb{N}$}\}.
\end{equation*}
Given $A\subset X$,  we define the \emph{forward $\mathscr{F}$-orbit}, or simply the $\mathscr{F}$-orbit for short, of $A$ by $$\OF(A)\eqdef\{h(x): \ h\in\langle\mathscr{F}\rangle^{+}, \ x\in A \}=\bigcup_{n\geq 1} F^n(A)$$
where $F$ is the natural extension of the Hutchinson operator associated with $\mathscr{F}$ to the powerset of $X$.
Recall that a subset $A\subset X$ is \emph{forward invariant}
for the semigroup $\langle\mathscr{F}\rangle^+$ if $f(A)\subset A$ for all $f\in\langle\mathscr{F}\rangle^+$.

\begin{defi}
We say that $\mathscr{F}$ is \emph{minimal} if every closed forward invariant set
for $\langle\mathscr{F}\rangle^+$ is either empty or coincides
with the whole space $X$.
Equivalently, $\mathscr{F}$ is minimal if and only if the $\mathscr{F}$-orbit
$\OF(x)$ is dense in $X$ for all $x\in X$.
\end{defi}


Note that if $X$ is a strict attractor of $\mathscr{F}$, then $\mathscr{F}$ is minimal since $F^{n}({x}) \to X$ in the Hausdorff metric for any $x\in X$. However, it is important to note that in general, a minimal IFS may not have a strict attractor.
The most trivial example of a minimal IFS that does not have a strict attractor is the IFS comprising of only one minimal map, such as an irrational rotation on the circle. In this case, the only candidate for an attractor of the IFS consisting of only one map is a singleton, while the minimal closed invariant set for a minimal map is the whole space.
However, it has been shown that the existence of an attractor in the IFS consisting of two minimal maps of the circle is possible, as established in~\cite{LSSV22}. Moreover, any minimal IFS of circle diffeomorphisms without a common invariant measure admits a strict attractor (see~\cite[Prop.~3.10]{BFMS16}).
Also, in~\cite[Prop.~3.8]{BFMS16}, it was proven that the whole space of any minimal quasi-symmetric IFS on a compact connected Hausdorff space is a strict attractor. By a quasi-symmetric IFS, we understand a finite family of continuous maps $\mathscr{F}$ with the property that there is $f \in \mathscr{F}$ such that its inverse map $f^{-1}$ belongs to $\mathscr{F}$. Finally, the following result, communicated to us by Sarizadeh~\cite{sarizadeh2023attractor}, provides another condition that guarantees a minimal IFS to have a strict attractor.


\begin{prop} \label{Sarizadeh} Let $\mathscr{F}$ be a finite family of continuous maps of a compact metric space $X$ to itself. If $\mathscr{F}$ is minimal and there exists $g \in \mathscr{F} $ with an attracting fixed point, then $X$ is a strict attractor of~$\mathscr{F}$.
\end{prop}
\begin{proof}
To prove that $X$ is a strict attractor, we need to show that $\dH(F^n(K),X)\to 0$ as $n\to \infty$ for every $K\in \mathscr{K}(X)$ where $F$ is the Hutchison operator associated with $\mathscr{F}$. To do this, it suffices to prove that $F^n(K)\cap V \not=\emptyset$ for any open set $V$ and for all $n$ large enough. Indeed, for any $\epsilon>0$, since $X$ is compact, we can cover $X$ by open balls $V_1,\dots,V_r$ of radius $\epsilon/2$. Since $F^n(K)\cap V_i\not=\emptyset$ for all $n$ large enough, we find $n_0\in\mathbb{N}$ such that $X \subset  B_\epsilon(F^n(K))$ and thus $\dH(F^n(K),X)<\epsilon$ for all $n\geq n_0$.

Let $V$ be an open set in $X$ and consider $K\in \mathscr{K}(X)$. Let $p\in X$ be the attracting fixed point of~$g$. Take a neighborhood $U$ of $p$ such that
\begin{equation*}
   g(\overline{U})\subset U \quad \hbox{
and} \quad  \bigcap_{m\geq 0}{g^{m}(U)}=\{p\}.
\end{equation*}

Since $\mathscr{F}$ is minimal,  there is $h\in \langle \mathscr{F}\rangle^+$ such that $h(p)\in V$. Analogously, given $x\in K$ there is $\phi\in \langle\mathscr{F}\rangle^+$
such that  $\phi(x)\in U$. Then $g^m\circ \phi(x)$
is close enough to $p$ and, by the continuity of $h$,  $h\circ g^m\circ \phi(x)\in V$
for all $m$ large enough. In particular, since
$ h\circ g^m\circ \phi(x) \in F^{t+m}(x) \subset F^{t+m}(K)$ where $t$ is the  number of generators involved in $\phi$ and $h$, we get that $F^n(K)\cap V\not= \emptyset$ for all $n$ large enough. This completes the proof of the proposition.
\end{proof}

Recall that the above example of a minimal IFS that does not admit a strict attractor is based on a family $\mathscr{F}$ consisting of only one element. The trivial way of increasing the number of generators by adding the identity map does not hold because then the whole space becomes a strict attractor. Therefore, one may ask if it is possible to find a minimal $\mathscr{F}$ that does not admit a strict attractor and such that the semigroup $\langle \mathscr{F}\rangle^+$ is free with at least two generators. In Corollary~\ref{exemplo-pablo}, we obtain such a non-trivial example as a consequence of Theorem~\ref{mainthm1}. This IFS is generated by two minimal circle rotations with different irrational rotation numbers.

\subsection{Robust minimality.} \label{sec:criterion_minimality}
Here, we will study when the property of minimality of an IFS persists under small perturbations. To do this, we first need to introduce a metric structure in the set of IFSs. Unless stated otherwise, in the sequel we will assume $(X,\D)$ to be a compact metric space.

Recall that a function {$f:X_1\to X_2$} between metric spaces $(X_i,\D_i)$, $i=1,2$ is said to be \emph{Lipschitz} if there is a constant $C\geq 0$ such that $\D_2(f(x),f(y))\leq C \D_1(x,y)$. The infimum of $C\geq 0$ such that the above inequality holds is called \emph{Lipschitz constant} and denoted by $\mathrm{Lip}(f)$.
The \emph{locally Lipschitz constant} of  $f$ at $z\in X$ is defined as
\begin{align*}
   Lf(z) = \inf_{r>0} \mathrm{Lip}(f|_{B(z,r)}) = \lim_{r\to 0^+} \mathrm{Lip}(f|_{B(z,r)}),
 \end{align*}
where
\begin{align*}
   \mathrm{Lip}(f|_{B(z,r)})=\sup\left\{   \frac{\D(f(x),f(y))}{\D(x,y)} \, : \ x,y\in B(z,r), \ x\not=y \right \}.
\end{align*}
We have that $\mathrm{Lip}(f)$ is greater than or equal to the supremum of $Lf(z)$ over all $z\in X$. If $f$ is a diffeomorphism of a Riemannian manifold $X$, then $Lf(z) = \|Df(z)\|$, i.e., the norm of the differential map of $f$ at $z$. We denote by
$\Lip(X)$  the set of  Lipschitz homeomorphisms of~$X$.
Furthermore, $\Lip(X)$ is a semigroup with the composition operation and a vector space with the usual function operations. We also endow this set  with the  distance
\begin{equation*} \label{def:dist-lip}
  \D_0(f,g)\eqdef \sup_{x\in X} \D(f(x),g(x)) + |Lf(x)-Lg(x)|, \qquad f,g\in \Lip(X).
\end{equation*}

In the following proposition, we will provide certain properties of the local Lipschitz constant $Lf(z)$ that we will use later.  Before going into the details, recall that a real-valued function $g$ defined on $X$ is called \emph{upper semicontinuous} if \(
\limsup_{x \rightarrow x_0} g(x)\leq g(x_0) \) for all $x_0\in X$,
or equivalently, the set $\{x \in X : g(x) < a\}$ is open for every real number $a$.

\begin{lem} \label{lem1} The map $Lf: X\to [0,\infty)$ is upper semicontinuous for any  Lipschitz map $f$ of $X$.  Moreover, for any pair of Lipschitz maps $f$ and $h$ of $X$, it satisfies the chain rule inequality
\[L(f \circ h)(x) \leq Lf(h(x))\cdot Lh(x), \quad \text{for every $x\in X$}.\]
\end{lem}
\begin{proof} Let us prove that $Lf$ is upper semicontinuous, that is, the set $E_a\eqdef \{x\in X :  Lf(x)<a\}$ is an open set for any $a>0$. To do this, for a fixed $z\in E_a$, by definition of  $Lf(z)$, we have $r>0$ such that
$$a>L_rf(z) \eqdef \mathrm{Lip}(f|_{B(z,r)})
\geq \sup_{x\in B(z,r)} Lf(x) .$$
Thus, $B(z,r) \subset E_a$ and consequently, $E_a$ is open.

Now we will prove the chain rule inequality. Since $h$ is a Lipschitz transformation and thus, a uniform continuous map,  given $\rho > 0$, there exists $r > 0$ such that if $\D(x, y) < r$, then $\D(h(x), h(y)) < \rho$. Hence,
\[
\frac{\D(f(h(x)), f(h(y)))}{\D(x, y)} = \frac{\D(f(h(x)), f(h(y)))}{\D(h(x), h(y))} \cdot \frac{\D(h(x), h(y))}{\D(x, y)} \leq L_\rho f(h(x)) \cdot L_r h(x).
\]
Taking $\rho$ and $r$ tending to $0^+$, we get the desired chain rule.
\end{proof}

Finally, we will introduce a distance in the set of IFSs  of  Lipschitz homeomorphisms as follows. First, we fix the number  $k\geq 1$ of maps that the IFS has.  We define
$$
  \bD(\mathscr{F},\mathscr{G})\eqdef \min_{\sigma \in S_k} \sum_{i=1}^k \D_0(f_i,g_{\sigma(i)}), \qquad \mathscr{F}=\{f_1,\dots,f_k\}, \ \ \mathscr{G}=\{g_1,\dots,g_k\},
$$
where $S_k$ is the permutation group on $\{1,\dots,k\}$. Notice that,  $\bD(\mathscr{F},\mathscr{G})=0$  imply that $\mathscr{F}$ is equal to $\mathscr{G}$ up to reorder the family. We say that  $\mathscr{G}$ is \emph{minimally ordered}
with respect to a family $\mathscr{F}$ if $$
\bD(\mathscr{F},\mathscr{G})= \sum_{i=1}^k \D_0(f_i,g_i).$$
Let us denote by $\IFS_k(X)$ the set of
families $\mathscr{F}=\{f_1,\dots,f_k\}$ with $f_i \in \Lip(X)$ for $i=1,\dots,k$,
where we identify families that are equal up to reordering its elements. It is not difficult to see that  $(\IFS_k(X),\bD)$ is a metric space.

\begin{defi} \label{def:robminimal} We say that $\mathscr{F} \in \IFS_k(X)$ is \emph{robustly minimal} if there exists a neighborhood $\mathcal{U}$ of $\mathscr{F}$ in $\IFS_k(X)$ such that  any $\mathscr{G}\in \mathcal{U}$ is minimal.
\end{defi}

A semigroup action generated by a family $\mathscr{F}$ of $C^1$ diffeomorphisms on a manifold $M$ was termed \emph{expanding} in~\cite{BFMS16} if, for any $x \in M$, there exists $h \in \langle\mathscr{F}\rangle^+$ such that $m(Dh^{-1}(x)) > 1$. In this context, $\mathrm{m}(T) = \|T^{-1}\|^{-1}$ denotes the co-norm of a linear operator $T$. Given the notation above, we can deduce that
\[
m(Dh^{-1}(x)) = \|Dh^{-1}(x)^{-1}\|^{-1} = \|Dh(h^{-1}(x))\|^{-1} = Lh(h^{-1}(x))^{-1}.
\]
Building on this, we introduce the subsequent definition that generalizes this concept to IFSs on broader metric spaces.


\begin{defi} \label{def:expading}
We say that $\mathscr{F}\in \IFS_k(X)$ is \emph{expanding} if
for every $x\in X$, there exists $h\in \langle\mathscr{F}\rangle^+$ such that $Lh(h^{-1}(x))<1$.
\end{defi}

\begin{obs}
Definition~\ref{def:expading} only requires that the elements of  $\mathscr{F}$ will be Lipschitz.
The terminology for this property has occasionally faced scrutiny, given that the expansion is observed in the backward action, rather than the forward action of $\mathscr{F}$. The justification for retaining the term "expanding" originates from the work of Shub and Sullivan~\cite{shub1985expanding}, who noted that expanding minimal group actions of $C^2$ diffeomorphisms on the circle are also ergodic with respect to the Lebesgue measure. The concept of an expanding group action in this setting is defined as follows (see~\cite[Def.~3.1]{deroin2020locally}): a group $G$ of diffeomorphisms on the circle $\mathbb{S}^1$ is expanding if, for every $x\in\mathbb{S}^1$, there exists $g\in G$ such that $|g'(x)|>1$. As $G$ can be generated as a semigroup, we have opted to retain the original terminology to maintain consistency.
\end{obs}

\begin{obs}
\label{obs-SH}
We want to emphasize that the expanding property of an IFS is intimately related to the property \textrm{SH} for partially hyperbolic diffeomorphisms introduced by Pujals and Sambarino in~\cite{PS06}. Even though in~\cite{PS06} the \textrm{SH} property was defined for the strong unstable foliation, we choose to present the analogous definition for the strong stable foliation here (see also~\cite[Def.~1.2]{pi2023hyperbolicity}). A $C^1$ diffeomorphism $f$ that is partially hyperbolic possesses the \emph{property \textrm{SH}} (indicative of some hyperbolicity on the central distribution) if there exists constants $\lambda <1$ and $C>0$ such that for any $x\in M$, there exists a $y\in W^{ss}_{loc}(x;f)$ for which
\[
  \|Df^{n}|_{E^c(f^{n-\ell}(y))}\|<C \lambda^n \quad \text{for all } n>0 \text{ and } \ell >0.
\]
To elucidate the connection between the property \textrm{SH} and expanding property for an IFS $\mathscr{F}$ of $C^1$ diffeomorphisms on a compact manifold $M$, we refer to the locally constant partially hyperbolic skew-product $\Phi=\tau\ltimes \mathscr{F}$, as introduced earlier in the introduction (also refer to \S\ref{sec:skew-prodcut}).
Here, the local strong stable set $W^{ss}_{loc}(\omega,x;\Phi)=W^{s}_{loc}(\omega;\tau)\times \{x\}$ is a direct product. Therefore, the property \textrm{SH} for $\Phi$ can be articulated as: for any $\xi\in \Sigma_k$ and $x\in M$, there exists $\omega \in W^s_{loc}(\xi;\tau)$ such that
\[
 \|Df_\omega^n(f^{n-\ell}_\omega(x))\| < C \lambda^n \quad \text{for all } n>0 \text{ and } \ell >0.
\]
Given a sufficiently large $n$, by choosing $\ell=2n$ and denoting $h=f^n_\omega \in\langle\mathscr{F}\rangle^+$, it follows that $\|Dh(h^{-1}(x)\|<1$, thereby inferring that $\mathscr{F}$ is expanding.
\end{obs}

\begin{obs}
The property \textrm{SH} is robust under small $C^1$ perturbations, as evidenced in~\cite{PS06}.
The same robustness holds for expanding IFSs under small perturbations, with respect to the metric $\bar{\D}_0$ (refer to Remark~\ref{obs-expanding-es-robusta}). On the other hand, according to~\cite{PS06}, if $f$ is a partially hyperbolic $C^1$ diffeomorphism with minimal strong unstable foliation and satisfying the property \textrm{SH} (as described above in Remark~\ref{obs-SH}),
then this foliation remains minimal after $C^1$ perturbations of $f$. An analogous result for IFS ensuring $C^1$ robust minimality can be found in~\cite[Thm.~A]{BFMS16}. Namely, the authors of~\cite{BFMS16} showed that any expanding and minimal finitely generated semigroup action of $C^1$ diffeomorphisms on a compact Riemannian manifold, is $C^1$-robustly minimal (i.e, the minimality persists by $C^1$ perturbation of the generators). In the subsequent discussion, we generalize this result to cover broader metric spaces, as presented in Theorem~\ref{teo-expanding-minimal-implica-mini-robusta} that follows.
\end{obs}


The following lemma gives a characterization of the expanding property.

\begin{lem} \label{lem:expading}
Any $\mathscr{F} \in \IFS_k(X)$ is expanding if and only if
there exist an open  cover $\{B_1,\dots,B_n\}$ of $X$, open sets $U_1,\dots, U_n$, constants
$\kappa>1$, $\epsilon>0$ and maps ${h}_{1},\dots,{h}_{n} \in \langle\mathscr{F}\rangle^+$ such that
\begin{equation} \label{eq:expanding1}
L{h}_{i}(z) < \kappa^{-1} \quad \text{for all $z\in U_i \supset B_\epsilon({h}^{-1}_i(B_i))$  and  $i=1,\dots,n$.}
\end{equation}
\end{lem}
\begin{proof}
Clearly, equation~\eqref{eq:expanding1} implies that $\mathscr{F}$ is expanding. For the converse, assume that $\mathscr{F}$ is expanding. For each $x\in X$, there exists a map $h \in \langle \mathscr{F}\rangle^+$ such that $Lh(h^{-1}(x))<\rho_x$ with $\rho_x<1$. By the upper semicontinuity of $Lh$ provided by Lemma~\ref{lem1}, and since $h^{-1}$ is continuous (as $h$ is a homomorphism), we have $Lh \circ h^{-1} < \rho_x$ in a neighborhood of $x$. In fact, we can take $\epsilon_x>0$ such that $Lh(z)<\rho_x$ for every $z\in B(h^{-1}(x),{2\epsilon_x})$. Considering the neighborhood  $B_x=h(B(h^{-1}(x),{\epsilon_x}))$ of $x$, by compactness of $X$, we get a cover of $X$ by open sets
$B_1,\dots,B_n$, constants $\kappa>1$, $\epsilon_1,\dots,\epsilon_n>0$, $\epsilon=\min\{\epsilon_1,\dots,\epsilon_n\}$  and maps
$h_1,\dots,h_n\in \langle \mathscr{F}\rangle^+$
such that~\eqref{eq:expanding1} holds with $U_i\eqdef B(h^{-1}_i(x_i),{2\epsilon_i})$
concluding the proof.
\end{proof}

\begin{obs}\label{obs-expanding-es-robusta}
A direct consequence of the preceding lemma is that the expanding property is robust, meaning it persists under small perturbations of $\mathscr{F}$ in $(\IFS_k(X),\bD)$. This is evident once we recognize that $\D_0$ controls both the $C^0$ distance and the local Lipschitz constant. Hence, there exists a neighborhood $\mathcal{U}$ of $\mathscr{F}$ in $(\IFS_k(X),\bD)$ wherein IFS in $\mathcal{U}$ has similar estimates as described in Lemma~\ref{lem:expading} for the cover $B_1,\dots,B_n$ of $X$. Explicitly, for any $\mathscr{G}\in\mathcal{U}$,
$$
L{g_i}(z)<\kappa^{-1}
\quad \text{for all $z\in U_i \supset B_{\sfrac{\epsilon}{2}}(g_i^{-1}(B_i))$, \ \  and \ \ $i=1,\dots,n$},
$$
where $g_i \in \langle\mathscr{G}\rangle^+$ is the perturbation of $h_i$  formed by the same finite sequence of generators assuming that $\mathscr{G}$ is minimally ordered with respect to $\mathscr{F}$.
\end{obs}

To generalize the criterion in~\cite{BFMS16} for obtaining robust minimality from the expanding property, we need a robustness of this property slightly different from that mentioned in Remark~\ref{obs-expanding-es-robusta}. For this reason, we need to consider the following class of metric space. A metric space $(X,\D)$ is called \emph{locally length space} if any point in $X$ admits a neighborhood $U$ such that the distance between two points $x,y\in U$ is the infimum of lengths of rectifiable paths from $x$ to $y$. Recall that  a path from $x$ to $y$ is a continuous curve \(\gamma : [0,1] \rightarrow X\) with $\gamma(0)=x$ and $\gamma(1)=y$ and it length is defined as
\[ \ell(\gamma) = \sup \left\{ \sum_{i=0}^{n-1} \D(\gamma(t_i),\gamma(t_{i+1})) \right\}, \]
where the supremum is taken over all partitions \( 0 = t_0 < t_1 < \ldots < t_n = 1 \) of the interval \([0,1]\). We will say that a curve \(\gamma\) is rectifiable if \( \ell(\gamma) < \infty \).
Euclidean spaces, Riemannian manifolds or more generally Finsler manifolds as well as every normed vector space are examples of locally length spaces.

\begin{prop} \label{prop:lengthspace} Let $(X,\D)$ be a compact locally length metric space. Then, there exists $R>0$ such that for any $0<r\leq R$ and any Lipschitz map $f:X\to X$, it holds that
$$
    \mathrm{Lip}(f|_{B(z,r)}) = \sup_{x\in B(z,r)} Lf(x) \quad \text{for all $z\in X$}.
$$
\end{prop}
\begin{proof}
Since $(X,\D)$ is locally length space, for each $z\in X$, there is a neighborhood $U$ of $z$ such that for any point in $x,y\in U$, \begin{equation} \label{eq:rectifica}
    \D(x,y)=\inf\ell(\gamma) \quad \text{where $\gamma:[0,1]\to X$ runs over the  rectifiable paths from $x$ to $y$.}
\end{equation}
Due to the compactness of \(X\), there exists an open covering \(\{U_1, \dots, U_n\}\) of \(X\) such that equation~\eqref{eq:rectifica} holds for any \(x,y \in U_i\), for \(i=1, \dots, n\). Let \(R > 0\) be a Lebesgue number of this cover. This means that every subset of \(X\) with diameter less than \(2R\) lies within some member of the cover. Thus, for any \(0 < r \leq R\) and any \(z \in X\), the ball \(B(z,r)\) is contained in some \(U_i\) for \(i \in \{1, \dots, n\}\), and hence satisfies equation~\eqref{eq:rectifica} for all \(x,y \in B(z,r)\).

Now, we follow~\cite[Lemma~2.2]{durand2010pointwise}. Fix \(z \in X\), \(0 < r \leq R\), and choose an arbitrary \(\varepsilon > 0\). Define
\[
   M\eqdef \sup_{x \in B(z,r)} Lf(x).
\]
For any rectifiable path \(\gamma: [0,1] \to X\) with \(\gamma(0) = x\) and \(\gamma(1) = y\), both in \(B(z,r)\), and for each \(t \in [0,1]\), there exists \(\rho_t > 0\) such that if \(p \in B(\gamma(t),\rho_t) \setminus \{\gamma(t)\}\), then
\[ \D(f(\gamma(t)), f(p)) \leq (M+\varepsilon) \D(\gamma(t),p). \]
Since \(\gamma\) is continuous, there exists \(\delta_t > 0\) such that
$(t - \delta_t, t + \delta_t) \subset \gamma^{-1}(B(\gamma(t),\rho_t))$.
This family of intervals is an open cover of $[0,1]$ and by compactness, we can choose a finite subcover, say
\(\{ I_{j} \}_{i=1}^{n}\), where $I_j=(t_j-\delta_j,t_j+\delta_j)$.
We may assume, refining the subcovering if necessary, that an interval \(I_i\) is not contained in \(I_j\) for \(i \neq j\).
If we relabel the indices of the points \(t_i\) in non-decreasing order, we can now choose a point \(s_{i,i+1} \in I_{i} \cap I_{i+1} \cap (t_i,t_{i+1})\) for each  $i=1,\dots, n-1$.
Using the auxiliary points that we have just chosen, we deduce that
\[ \D(x,\gamma(t_1)) + \sum_{i=1}^{n-1} \left( \D(\gamma(t_i),\gamma(s_{i,i+1})) + \D(\gamma(s_{i,i+1}),\gamma(t_{i+1})) \right) + \D(\gamma(t_n),y) \leq \ell(\gamma), \]
and so
\[ \D(f(x), f(y)) \leq (M+\varepsilon) \ell(\gamma). \]
Finally, since this is valid for all \(\varepsilon > 0\) and rectifiable path $\gamma$ from $x$ to $y$, we conclude that
$\D(f(x), f(y))\leq M \D(x,y)$. This implies that $\mathrm{Lip}(f|_{B(z,r)})\leq M$. The other inequality always holds, hence we conclude the equality and complete the proof.
\end{proof}

As a corollary, we get the robustness of the expanding property that we need to establish the criterion to obtain robust minimality.

\begin{cor} \label{cor:expanding} Let $(X,\D)$ be a compact, locally length space and consider $\mathscr{F}\in \IFS_k(X)$. Then, $\mathscr{F}$ is expanding if and only if
there exist a neighborhood $\mathcal{U}$ of $\mathscr{F}$, an open cover $\{B_1,\dots,B_n\}$ of $X$, open sets $U_1,\dots, U_n$ and  constants
$\kappa>1$, $\epsilon>0$   such that for any $\mathscr{G}\in \mathcal{U}$, there are
 ${h}_{1},\dots,{h}_{n}\in \langle\mathscr{G}\rangle^+$ satisfying
\begin{equation} \label{ii.9bis}
  \mathrm{d}({h}_{i}(x),{h}_{i}(y))<\kappa^{-1}\mathrm{d}(x,y), \ \
  \text{for all} \ \ x,y\in U_i \supset B_\epsilon({h}^{-1}_{i}(B_{i})) \ \text{and} \ \ i=1,\dots,n.
  \end{equation}
\end{cor}

\begin{proof}
If condition~\eqref{ii.9bis} is satisfied, then it is immediate that \(\mathscr{F}\) is expanding. For the converse, assume \(\mathscr{F}\) is expanding. Referring to Remark~\ref{obs-expanding-es-robusta}, the expanding property of \(\mathscr{F}\) ensures the robustness of~\eqref{eq:expanding1}. From the proof of Lemma~\ref{lem:expading}, we note that each \(U_i\) is an open ball of radius \(2\epsilon_i > 0\) for \(i = 1,\dots,n\). Moreover, \(\epsilon_i\) can be chosen arbitrarily small. Specifically, we can select \(\epsilon_i\) such that \(0 < \epsilon_i \leq R/2\), where \(R > 0\) is provided by Proposition~\ref{prop:lengthspace}. Thus, from Proposition~\ref{prop:lengthspace}, it follows that \(\mathrm{Lip}(h_i|_{U_i}) = \sup_{x \in U_i} Lh_i(x) < \kappa^{-1}\). This implies~\eqref{ii.9bis}, completing the proof.
\end{proof}

Now,  we  conclude that robust minimality is a consequence of being both expanding and minimal:

\begin{teo}\label{teo-expanding-minimal-implica-mini-robusta}
Let $(X,\D)$ be a compact locally length metric space, and consider $\mathscr{F}\in \IFS_k(X)$. If $\mathscr{F}$ is expanding and minimal, then $\mathscr{F}$ is robustly minimal.
\end{teo}
\begin{proof}
Let $\mathscr{F} \in \IFS_k(X)$ be expanding and minimal.
By Corollary~\ref{cor:expanding}, there exist a neighborhood $\mathcal{U}$ of $\mathscr{F}$, a cover of $X$ by open sets $B_{1},\dots,B_{n}$, open sets $U_1,\dots, U_n$ and  constants
$\kappa>1$, $\epsilon>0$  such that for any $\mathscr{G}\in U$, there are maps
 ${h}_{1},\dots,{h}_{n}\in \langle\mathscr{G}\rangle^+$ satisfying~\eqref{ii.9bis}. Since $X$ is compact, there is a Lebesgue number $\delta>0$ of the open cover $\{B_i:i=1,\dots,n\}$.
 Take $0<\varepsilon<\min\{\delta,\epsilon\}$.
\begin{afir} \label{claim:epsilon-minimal}
 There is a neighborhood
$\mathcal{V}$ of $\mathscr{F}$ in $\IFS_k(X)$ such that for
any $\mathscr{G}\in\mathcal{V}$ the $\mathscr{G}$-orbits are $\varepsilon$-dense.  That is, $\OG(x)\cap B\not=\emptyset$
for all $x\in X$ and  open ball $B$ of radius $\varepsilon$.
\end{afir}
\begin{proof}
By the compactness of $X$, we can cover it with finitely many open balls $W_1,\dots,W_n$ of radius $\varepsilon/4$. Since $\mathscr{F}$ is minimal, for each $x\in X$, there are $g_1,\dots, g_n\in \langle\mathscr{F}\rangle^+$ such that $g_i(x) \in W_i$ for all $i=1,\dots,n$. Due to the continuity of $g_i$ and the compactness of $X$, we can find a finite open cover $V_1,\dots,V_m$ of $X$ and maps $g_{ij}\in \langle\mathscr{F}\rangle^+$, where $i=1,\dots,n$ and $j=1,\dots,m$, such that $\overline{g_{ij}(V_j)} \subset W_i$. As this property only involves finitely many maps and open sets, we can find a neighborhood $\mathcal{V}$ of $\mathscr{F}$ in $\IFS_k(X)$ where it remains valid. This means that for any $\mathscr{G}\in \mathcal{V}$, the $\mathscr{G}$-orbit of $x \in V_j$ visits $W_i$ for all $i=1,\dots,n$ and $j=1,\dots,m$. Since any ball $B$ with radius $\varepsilon$ contains at least one of the balls  ${W_1,\dots W_n }$ and ${V_1,\dots V_m}$ covers $X$, the $\mathscr{G}$-orbit of any point in $X$ is $\varepsilon$-dense as required.
\end{proof}

Consider the neighborhood $\mathcal{W}=\mathcal{V}\cap \mathcal{U}$ of $\mathscr{F}$.

\begin{afir} \label{claim:deep} Let $i\in\{1,\dots,n\}$ and $\mathscr{G} \in \mathcal{W}$. Take $y\in B_{i}$ and consider $r>0$ such that $B(y,r)\subset B_{i}$ with $r\leq \kappa^{-1}\varepsilon$. Then $h_{i}(B(h^{-1}_i(y),\kappa r))\subset B(y,r)$.
\end{afir}
\begin{proof}
Denote $y_0=h^{-1}_i(y) \in h_i^{-1}(B_i)$.
Since $\kappa r \leq \varepsilon<\epsilon$, we have  $B(y_0,\kappa r) \subset B_{\epsilon}(h_i^{-1}(B_i)) \subset  U_i$. Now,
if $x\in B(y_0,\kappa r)$,  by~\eqref{ii.9bis},
$$   \D(h_i(x),y)=\D(h_i(x),h_i(y_0))<\kappa^{-1}\D(x,y_0)<\kappa^{-1}
   \cdot \kappa r=r.
$$
Therefore, we conclude that $h_{i}(B(y_0,\kappa r))\subset B(y,r)$ and complete the proof of the claim.
\end{proof}

Finally, we will prove the density of the orbits
$\OG(x)$ for all $x\in X$. According to Claim~\ref{claim:epsilon-minimal} and since $\mathscr{G} \in \mathcal{W}$,  given $x\in X$  we have that the orbit $\OG(x)$ is $\varepsilon$-dense.
Now, we will show that actually
$\OG(x)$ is $\kappa^{-1}\varepsilon$-dense in $X$. To do this, take any $y\in X$ and consider the open ball $B(y,\kappa^{-1}\varepsilon)$.
Since $\kappa^{-1}\varepsilon<\delta$, then there is
$i\in\{1,2,\dots,n\}$ such that $B(y,\kappa^{-1}\varepsilon)\subset B_{i}$.
Then, 
by Claim~\ref{claim:deep} we have
\begin{equation*}
  h_{i}(B(h_i^{-1}(y),\kappa{\kappa}^{-1}\varepsilon))\subset B(y,\kappa^{-1}\varepsilon).
\end{equation*}
Since $\OG(x)$ is $\varepsilon$-dense, there is $h\in\langle\mathscr{G}\rangle^+$
such that $h(x)\in B(h_i^{-1}(y),\varepsilon)=B(h_i^{-1}(y),\kappa{\kappa}^{-1}\varepsilon)$.
Therefore, $(h_{i}\circ h)(x)\in B(y,\kappa^{-1}\varepsilon)$. Consequently,
the orbit $\OG(x)$ is $\kappa^{-1}\varepsilon$-dense.


Repeating this argument, we can show that the orbit $\OG(x)$
is $\kappa^{-p}\varepsilon$-dense for all $p\in\mathbb{N}$.
Since $\kappa^{-p}\varepsilon\to 0$ as $p\to\infty$,
we find that any orbit $\OG(x)$ is dense.
That is, $\mathscr{G}$ is minimal.
\end{proof}

\subsection{Blending regions.}\label{subsec: Blending-region}

The main tool for building robustly minimal families is
known as \emph{globalized blending region}.

\begin{defi} \label{globalization}
An open set $B$ is called \emph{globalized} for $\mathscr{F}$ if there exist maps
$T_1,\dots,T_s, S_1\dots,S_t \in\langle\mathscr{F}\rangle^+$ such that 
\begin{equation*}
    X=T_1(B)\cup \dots\cup T_s(B) = S_1^{-1}(B)\cup\dots \cup S_t^{-1}(B).
\end{equation*}
\end{defi}

\begin{obs} \label{rem:backward-globalized}
If $X$ is compact and  $\mathscr{F}$ minimal, then any open set $B$ of $X$ is backward globalized, i.e., there exist maps
$S_1\dots,S_t \in\langle\mathscr{F}\rangle^+$ such that
$    X= S_1^{-1}(B)\cup\dots \cup S_t^{-1}(B)$.
Actually, it is not difficult to see that the converse also holds.
\end{obs}


\begin{defi} \label{def:blending-region}
An open set $B\subset X$ is called \emph{blending region}
for $\mathscr{F}$ if there exist $h_1,\dots,h_m\in\langle\mathscr{F}\rangle^+$ and an open set
$D\subset X$ such that  $\overline{B}\subset D$ and
\begin{enumerate}[label=(\arabic*)]
\item\label{item1-defi-region-misturadora}  $\overline{B}\subset h_1(B)\cup \dots \cup h_m(B)$, 
\item\label{item2-defi-region-misturadora} $h_i: \overline{D}\to D$ is a contracting map for
$i=1,\dots,m$.
\end{enumerate}
\end{defi}

In a general metric space $(X,\D)$, having a globalized blending region is not a robust property in $(\IFS_k(X),\bar{\mathrm{d}}_0)$. Small perturbations of $\mathscr{F}$ in $\IFS_k(X)$ can destroy the open covers of both $X$ and $\overline{B}$. However, the following proposition demonstrates that this is not the case if $X$ is compact.

\begin{prop}\label{propo-blending-region-robust}
Assume that $(X,\D)$ is a compact metric space, and let $B$ be a globalized blending region for $\mathscr{F}\in \IFS_k(X)$. Then, there exists a neighborhood $\mathcal{U}$ of $\mathscr{F}$ in $\IFS_k(X)$ such that for any $\mathscr{G}\in \mathcal{U}$, $B$ is also a globalized blending region for $\mathscr{G}$.
\end{prop}
\begin{proof}
We first note that the condition~\ref{item2-defi-region-misturadora} in Definition~\ref{def:blending-region} is robust in $(\IFS_k(X),\bar{\mathrm{d}}_0)$, as $\D_0$ controls the $C^0$ distance and the locally Lipschitz constant. Next, we will prove that the globalization property $X=T_1(B)\cup \dots \cup T_s(B)$ is also robust. The proof of the robustness of any other cover in Definitions~\ref{globalization}~and~\ref{def:blending-region} is analogous.

Since $(X,\D)$ is compact, by the Lebesgue's number lemma, there exists a Lebesgue number $\delta>0$ for the open cover $\{T_j(B): j=1,\dots,s\}$. We cover $X$ by finitely many open balls $B(x_i,\delta/2)$, $i=1,\dots,n$, and for each $i$, there exists $j_i$ such that $B(x_i,\delta)\subset T_{j_i}(B)$. If the boundary $\partial B(x_i,\delta/2)$ is empty, then $B(x_i,\delta/2) \subset \tilde{T}(B)$ for any sufficiently small perturbation $\tilde{T}\in \Lip(X)$ of $T_{j_i}$. We can assume that $\partial B(x_i,\delta/2)$ is not empty for any $i$, and let $$\eta=\min_{1\leq i\leq n} \D(\partial B(x_i,\delta/2),\partial T_{j_i}(B))>0.$$ Moreover, $\eta=\eta(\mathscr{F})$ varies continuously with respect to $\mathscr{F}$ in $\IFS_k(X)$. Thus, we can find a neighborhood $\mathcal{U}$ of $\mathscr{F}$ in $\IFS_k(X)$ such that $\eta(\mathscr{G})>0$ for all $\mathscr{G}\in \mathcal{U}$. In particular, $B(x_i,\delta/2) \subset \tilde{T}_{j_i}(B)$ for each $i=1,\dots,n$, where $\tilde{T}_{j_i}$ is the perturbation of $T_{j_i}$ formed by the same finite sequence of generators, assuming that $\mathscr{G}$ is minimally ordered with respect to $\mathscr{F}$. Since $\{B(x_i,\delta/2): i=1,\dots,n\}$ covers $X$, we have $X=\tilde{T}_1(B)\cup \dots \cup \tilde{T}_s(B)$, completing the proof.
\end{proof}

The following theorem provides possibly the simplest criterion for constructing
robustly minimal families.

\begin{teo}
\label{prop1-blending region}
Let $(X,\D)$ be a compact metric space. Then, any $\mathscr{F}\in \IFS_k(X)$  with a globalized blending region is  robustly minimal.
\end{teo}
\begin{proof}
Consider $\mathscr{F}\in \IFS_k(X)$ and let $\langle\mathscr{F}\rangle^+$ denote the semigroup generated by $\mathscr{F}$. By hypothesis, there exists a globalized blending region $B$ for $\mathscr{F}$. In fact, Proposition~\ref{propo-blending-region-robust} implies that $B$ is also a globalized blending region for any $\mathscr{G}$ belonging to a small neighborhood $\mathcal{U}$ of $\mathscr{F}$ in $\IFS_k(X)$.
From Definition~\ref{def:blending-region}, there are $h_{1},h_{2},\dots,h_{m}\in\langle\mathscr{G}\rangle^+$ and an open set $D\subset X$ such that $\overline{B}\subset D$ and conditions~\ref{item1-defi-region-misturadora} and~\ref{item2-defi-region-misturadora} are satisfied.

\begin{afir}\label{afirmacion-teo-blending-rigion-globali-implica-robus-minim}
 We have that $B\subset\overline{\OG(x)}$ for all $x\in B$.
\end{afir}
\begin{proof}
To see this, we will use Hutchison's theorem~\cite{HUT81}, which states that the Hutchinson operator $F_{\mathscr{H}}:\mathscr{K}(\overline{D})\to \mathscr{K}(\overline{D})$ associated with the family of contractions $\mathscr{H}=\{h_{1},\dots,h_{m}\}$ of $\overline{D}$ is also a contraction. As a consequence, there exists a unique compact set $A\subset \overline{D}$ such that \begin{equation*}
   A=F_{\mathscr{H}}(A)=\bigcap_{n\in\mathbb{N}} F_{\mathscr{H}}^n(\overline{D})
\quad \text{and} \quad
A \subset \overline{\mathcal{O}_{\mathscr{H}}^+(x)} \ \ \text{for all $x\in\overline{D}$.}
\end{equation*}
Since $B\subset\overline{D}$ and $B\subset F_{\mathscr{H}}(B)$, it follows that $B\subset A$. Consequently, we have $B \subset \overline{\mathcal{O}_{\mathscr{H}}^+(x)}$ for all $x\in\overline{D}$.
Finally, as $\langle\mathscr{H}\rangle^+$ is a subsemigroup of $\langle\mathscr{G}\rangle^+$, we can conclude that for every $x\in B$, it holds that $B\subset\overline{\OG(x)}$, as required.
\end{proof}

To show that $\mathscr{G}$ is minimal, i.e., $X=\overline{\OG(y)}$ for every $y\in X$, we note that since $B$ is globalized for $\mathscr{G}$, there exist $T_1,\dots,T_s, S_1,\dots,S_t \in\langle\mathscr{G}\rangle^+$ such that
\begin{equation*}
X=T_1(B)\cup \dots\cup T_s(B) = S_1^{-1}(B)\cup\dots \cup S_t^{-1}(B).
\end{equation*}
Given $y\in X$ and any open set $U$ of $X$, we can find $i\in\{1,\dots,s\}$ and $j\in \{1,\dots,t\}$ such that $y \in S_{j}^{-1}(B)$ and $U\cap T_i(B)\neq\emptyset$. Therefore, we have $S_j(y)\in B$ and $T_i^{-1}(U)\cap B$ is an open set in $B$.

Using Claim~\ref{afirmacion-teo-blending-rigion-globali-implica-robus-minim}, we obtain $B\subset\overline{\OG(S_{j}(y))}$. In particular, we have $T_i^{-1}(U)\cap B\subset\overline{\OG(S_{j}(y))}$. Hence, there is $L\in\langle\mathscr{G}\rangle^+$ such that $L\circ S_j(y) \in T_i^{-1}(U)\cap B$, and so $T_i \circ L \circ S_j(y) \in U$. This proves that the orbit $\OG(y)$ is dense in $X$, and therefore $\mathscr{G}$ is minimal.
\end{proof}

The following theorem shows that it is possible to construct a minimal IFS from weaker assumptions than the globalized blending region. Before stating the next result, we indicate that the notation $\mathrm{diam}(A)$ represents the diameter of a set $A$, i.e., the largest distance between pairs of points in $A$.
Similar to the forward $\mathscr{F}$-orbit of $A$, we introduce the \emph{backward $\mathscr{F}$-orbit} of $A$ denoted by $\mathcal{O}_{\mathscr{F}}^-(A)$. It is defined as the union of the inverse images of $A$ under each element in $\langle\mathscr{F}\rangle^+$, that is, $\mathcal{O}_{\mathscr{F}}^-(A) = \{f^{-1}(x): f\in\langle\mathscr{F}\rangle^+, x\in A\}$.


\begin{teo} \label{thm-BFMS16} Let $(X,\D)$ be a metric space and consider $\mathscr{F} \in \IFS_k(X)$.
Let $B$ be  an open set of $X$ such that
$
 X= \overline{\OF(B)} = \mathcal{O}_{\mathscr{F}}^-(B)$.
Suppose that there are functions $\{h_i\}_{i\geq 1} \subset \langle\mathscr{F}\rangle^+$ such that
\begin{equation}\label{ii.10}
 B\subset \bigcup_{i=1}^\infty h_i (B) \ \  \text{and} \ \
 \lim_{n\to\infty} \mathrm{diam}(h_{i_1}\circ \dots \circ h_{i_n}(B))=0,
\end{equation}
for any sequence $\{i_n\}_{n\geq 1}$ with $i_n \in \mathbb{N}$ and
$h_{i_1}\circ \dots \circ h_{i_n}(B)\cap B\not=\emptyset$.
Then  $\mathscr{F}$ is 
minimal.
\end{teo}
\begin{proof}
We begin by presenting the following lemma, which will be used to prove the theorem.

\begin{lem}
\label{lem-densidad-A} Let $B$ be an open set that satisfies~\eqref{ii.10}.
Then for every $x\in B$,
there is a sequence $\{i_n\}_{n\geq 1}$, $i_n \in \mathbb{N}$ such that
\begin{equation*}
   x=\lim_{n\to\infty} h_{i_1}\circ\dots\circ h_{i_n}(y), \ \ \text{for all} \ y\in B.
\end{equation*}
\end{lem}
\begin{proof}
We will recursively define the sequence $\{i_n\}_{n\geq 1}$.
Once found an integer $i_n>0$ such that
\begin{equation*}
  x\in h_{i_1}\circ\cdots \circ h_{i_n}(B)
\end{equation*}
we deduce from~\eqref{ii.10} that
\begin{equation*}
  x\in h_{i_1}\circ\cdots\circ h_{i_n}(B)\subset \bigcup_{i=1}^{\infty}
  h_{i_1}\circ\cdots \circ h_{i_n}\circ h_i(B).
\end{equation*}
Thus, we can find $i_{n+1}$ such that
$x\in h_{i_1}\circ\cdots \circ h_{i_n}\circ h_{i_{n+1}}(B)$.
Hence, we have built the sequence $\{i_n\}_{n\geq 1}$ such that
\begin{equation*}
 x\in \bigcap_{n\geq 1} h_{i_1}\circ \dots \circ h_{i_n}(B).
\end{equation*}
Furthermore, since $x\in h_{i_1}\circ\cdots \circ h_{i_n}(B)$ for
all $n\in \mathbb{N}$, we deduce that for any $y\in B$, it holds that
\begin{equation*}
   \mathrm{d}(h_{i_1} \circ \dots \circ h_{i_n}(y), x) \leq
\mathrm{diam}(h_{i_1}\circ \dots \circ h_{i_n}(B)).
\end{equation*}
Note that, since $x\in B$ and $x\in h_{i_1}\circ\dots\circ h_{i_n}(B)$,
then the sequence $\{i_n\}_{n\geq 1}$ satisfies that
$h_{i_1}\circ\dots\circ h_{i_n}(B)\cap B\not=\emptyset$ for all $n\in\mathbb{N}$.
Thus, by hypothesis $\lim_{n\to\infty}{\mathrm{diam}(h_{i_1}\circ\dots\circ h_{i_n}(B))}=0$.
This shows that $\lim_{n\to\infty}{\mathrm{d}(h_{i_1}\circ\dots\circ h_{i_n}(y), x)}=0$ proving the lemma.
\end{proof}

 Now, let us proceed with the proof of the theorem. Given a point
 $x\in X$ and an open set $U$ of $X$, since
$X=\overline{\OF(B)} = \mathcal{O}_{\mathscr{F}}^-(B)$, we have
there are $T,S\in \langle \mathscr{F} \rangle^+$ such that
$S(x) \in B$ and $B\cap T^{-1}(U)$ is open not empty.
By Lemma~\ref{lem-densidad-A}, any point in $B$ can be approximated
by some iterated of the form $h_{i_1}\dots \circ h_{i_n}(S(x))$.
In particular, there is $h\in\langle\mathscr{F}\rangle^+$ such that
$h\circ S(x)\in T^{-1}(U)\cap B$.
Thus, $T\circ h \circ S(x)\in U$.
This proves that the orbit $\OF(x)$ is dense in $X$ and
therefore, $\mathscr{F}$ is minimal.
\end{proof}

In the particular case that the ambient metric space is a compact locally length space, we can obtain the following slight improvement of Theorem~\ref{prop1-blending region}.

\begin{cor} \label{thm-BFMS16-manifold} Let $(X,\D)$ be a compact, locally length  space and consider $\mathscr{F}\in \IFS_k(X)$. Consider  a globalized open set $B$ for $\mathscr{F}$ and  assume that there are an open set $D$,  $0<\beta<1$ and $\{h_i\}_{i\geq 1} \subset \langle\mathscr{F}\rangle^+$ such that $B\subset D$,  $Lh_i|_{D} \leq \beta$ and $h_i(D)\subset D$ for all $i\geq 1$, and
\begin{equation}\label{ec.blender}
B\subset \bigcup_{i=1}^\infty h_i (B).
\end{equation}
Then  $\mathscr{F}$ is expanding and  robustly
minimal.
\end{cor}
\begin{proof}
First, notice that since $B$ is a globalized open set for $\mathscr{F}$, it holds that $B\subset \overline{\OF(B)} = \mathcal{O}_{\mathscr{F}}^-(B)$. Hence, to apply Theorem~\ref{thm-BFMS16},  we will show that, for any $\{i_n\}_{n\geq 1}$ it holds
that
\begin{equation} \label{ii.11}
\mathrm{diam}(h_{i_1}\circ\dots \circ h_{i_n}(B))\to 0 \quad
\text{as $n\to\infty$}.
\end{equation}
To prove this convergence observe that
any composition
$h_{i_1}\circ\dots \circ h_{i_n}$ sends  $D$ to $D$ with contraction rate
$\beta^n$. Hence, $\mathrm{diam}(h_{i_1}\circ \dots h_{i_n}(B)) \leq \beta^n \mathrm{diam}(B)$
and consequently~\eqref{ii.11} holds. Therefore,
by Theorem~\ref{thm-BFMS16}, $\mathscr{F}$ is minimal. To conclude the proof, we will show that $\mathscr{F}$ is expanding. Thus, by Theorem~\ref{teo-expanding-minimal-implica-mini-robusta} we have that $\mathscr{F}$ is robustly minimal.

 Given any $x\in X$, for the globalized property in Definition~\ref{globalization}, we have $T^{-1}_j(x)\in B$ for some $j\in \{1,\dots,m\}$.  Then, by the cover~\eqref{ec.blender}, we can iterate $T_j^{-1}(x)$ by some
 $h_i^{-1}$ and remain in $B$. That is, $h^{-1}_i\circ T_j^{-1}(x)\in B$.
 Since $Lh_i|_B \leq \beta$, by Lemma~\ref{lem1}, we have that
$$
 L(T_j\circ h_{i})((h_{i}^{-1}\circ T_j^{-1})(x)) \leq Lh_{i}(h_{i}^{-1}\circ T^{-1}_j(x)) \cdot LT_{j}(T^{-1}_j(x)) \leq \beta LT_{j}(T^{-1}_j(x)).
$$
Since $\beta<1$, repeating the argument the number of times necessary, we found $g\in\langle\mathscr{F}\rangle^+$
 such that
 $Lg(g^{-1}(x))<1$ and thus $\mathscr{F}$ is expanding.
\end{proof}


We will apply the above result to construct
examples of robustly minimal IFS comprising of diffeomorphisms of the
circle.  Namely, we will provide an alternative proof of the following theorem from \cite[Theorem B]{BFS14}, which generalizes earlier results on robustly minimal IFSs on the circle described in \cite{GI99, GI00}.

\begin{teo}\label{teo-B-BFS14}
Let $f_1, f_2$ be preserving orientation circle diffeomorphisms. Assume that $f_1$ is an irrational rotation and
that $f_2$ is not a rotation. Then
$\mathscr{F}=\{f_1,f_2\}$ is expanding and robustly minimal.
\end{teo}
\begin{proof}
Since $f_2$ is a circle diffeomorphism,
$$
1=\mathrm{Leb}(\mathbb{S}^1)=\mathrm{Leb}((f_2)^{-1}(\mathbb{S}^1))=\int |(f_2^{-1})'(z)| \, d\mathrm{Leb}(z),
$$
where $\mathrm{Leb}$ is the normalized Lebesgue measure on $\mathbb{S}^1$. Moreover, since $f_2$ is not a rotation, from the above estimate,  there exists a point $z\in \mathbb{S}^1$ such that $|(f_2^{-1})'(z)|>1$. From here, we have an open interval $I$ containing $z$ such that $|(f_2^{-1})'|_I|>1$. On the other hand, since $f_1$ is an irrational rotation,  $f_1^{-1}$ is minimal and thus, for every $x\in \mathbb{S}^1$, there exists $n\in \mathbb{N}$ such that $f_1^{-n}(x)\in I$. Then,
$$|(g^{-1})'(x)|=|(f_2^{-1})'(f_1^{-n}(x))|\cdot |(f_1^{-n})'(x)|=|(f_2^{-1})'(f_1^{-n}(x))|>1,$$  where $g=f_1^{n}\circ f_2 \in \langle \mathscr{F}\rangle^+$.
This implies that $\mathscr{F}$ is expanding since it is also minimal (since $f_1$ is an irrational rotation), we conclude from Theorem~\ref{teo-expanding-minimal-implica-mini-robusta}  that $\mathscr{F}$ is also robustly minimal.
\end{proof}

\begin{obs} \label{rem:back-exp-min}
Note that the above theorem may also be applied to $\mathscr{F}^{-1}=\{f_1^{-1}, f_2^{-1}\}$, obtaining that $\mathscr{F}^{-1}$ is also expanding and robustly minimal. Using the terminology introduced in~\S\ref{IFS-intro}, we can say that the assumptions of Theorem~\ref{teo-B-BFS14} for $\mathscr{F}=\{f_1,f_2 \}$ imply that $\mathscr{F}$ is, both,  forward and backward expanding and robustly forward and backward minimal.
\end{obs}

\section{Transitivity of IFS}\label{sec:transitividad}

We start by giving some a priori different notion of transitivity. This property was study~\cite{koropecki2009transitivity} when  the generators
are measure preserving maps.

\begin{defi}\label{definicion-transitividad}
We say that family $\mathscr{F}$ of finitely many continuous maps of $X$ is
\begin{enumerate}[leftmargin=1.25cm]
    \item[\mylabel{item1-teo-equiv-transitividad}{\textrm{(T)}}] \emph{transitive} if the forward $\mathscr{F}$-orbit of every open set is dense on $X$.
    \item[\mylabel{item2-teo-equiv-transitividad}{\textrm{(TT)}}] \emph{topologically transitive} if for every pair of nonempty open sets $U,V\subset X$, \\
    there exists $h\in\langle \mathscr{F} \rangle^{+}$ such that $h(U)\cap V\neq\emptyset$.
\end{enumerate}
\end{defi}
The above notions of transitivity are equivalent, as discussed below.

\begin{teo}\label{teo-equiva-transitividad}
Let $\mathscr{F}$ be a finite family of continuous maps of $X$ to itself. Then,
$\mathscr{F}$ is transitive if and only if it is topologically transitive.
\end{teo}
\begin{proof}
To deduce~\ref{item2-teo-equiv-transitividad} from~\ref{item1-teo-equiv-transitividad}, take open sets $U,V\subset X$.
Since $\mathscr{F}$ is transitive, we have that the $\mathscr{F}$-orbit of $U$ is dense on $X$ and
then there is $h\in\langle\mathscr{F}\rangle^{+}$ such that $h(U)\cap V\neq\emptyset$. Thus, we get~\ref{item2-teo-equiv-transitividad}.

Now suppose~\ref{item2-teo-equiv-transitividad} and take an open set $U\subset X$ and $y\in X$.
Given $n\in\mathbb{N}$ consider the open ball $B(y,{1}/{n})$. Since $\mathscr{F}$ is~\ref{item2-teo-equiv-transitividad},
there is $h_{n}\in\langle\mathscr{F}\rangle^{+}$ such that
$B(y,{1}/{n})\cap h_{n}(U)\neq\emptyset$. Take $z_{n}\in B(y,{1}/{n})\cap h_{n}(U)$. Note that
$z_{n}\in h_{n}(U)$  and $z_{n}\to y$ as $n\to\infty$. Thus,
\begin{equation*}
y\in\overline{\{h(U): \ h\in\langle\mathscr{F}\rangle^{+}\}}=\overline{\OF(U)}.
\end{equation*}
Therefore, the $\mathscr{F}$-orbit of $U$ is dense in $X$ and consequently we have~\ref{item1-teo-equiv-transitividad}.
\end{proof}

\begin{obs} \label{rem:backward-trans}
    If $\mathscr{F}$ consists of homeomorphisms of $X$, we can consider the inverse family given by $\mathscr{F}^{-1}=\{ f^{-1}: f\in \mathscr{F}\}$. In such case, it immediately follows that $\mathscr{F}$ is transitive if and only if $\mathscr{F}^{-1}$ is~transitive.
\end{obs}

\begin{obs} \label{rem:min-tran}
    Clearly, if $\mathscr{F}$ is minimal, then $\mathscr{F}$ is also transitive.
\end{obs}

The following result asserts that, under specific conditions, the concept of transitivity is equivalent to the existence of dense orbits. A similar result in the context of group actions can be found in~\cite[Prop. 1]{cairns2007topological}.

\begin{prop}\label{propo-equivalencia-transitividad-densidad-puntual}
Let $X$ be a second countable Baire space, and let $\mathscr{F}$ be a finite family of homeomorphisms of $X$.
Then, $\mathscr{F}$ is transitive if and only if there exists a (residual) dense set in $X$ where each point within it has a dense forward $\mathscr{F}$-orbit in $X$. Moreover, $\mathscr{F}$ is transitive if and only if there exists $x\in X$ such that the forward and backward $\mathscr{F}$-orbits of $x$ are both dense in $X$.
\end{prop}
\begin{proof}
Assume that $\mathscr{F}$ is transitive. By Theorem~\ref{teo-equiva-transitividad}, we conclude that $\mathscr{F}$ is topologically transitive.
Let $\mathcal{B}=\{U_{i}:~i\in\mathbb{N}\}$ be a countable base for $X$, and consider $V_{i}\eqdef\mathcal{O}_{\mathscr{F}}^-(U_{i})=\bigcup_{f\in \langle\mathscr{F}\rangle^+} f^{-1}(U_{i})$. Take an open set $A\subset X$ and fix $i\in\mathbb{N}$. Then, there exists $h\in\left \langle \mathscr{F} \right \rangle^{+}$ such that $h^{-1}(U_{i})\cap A\neq\emptyset$. Hence, $V_{i}$ is dense in $X$.
Since $X$ is a Baire space, we have $V=\bigcap_{i\in\mathbb{N}}{V_{i}}$ is dense in $X$. We claim that each element of $V$ has a dense forward $\mathscr{F}$-orbit. Indeed, if $W\subset X$ is open, one has $U_{j}\subset W$ for some $j\in\mathbb{N}$. Thus, if $x\in V$, we have $x\in V_{j}$, and it follows that $\tilde{h}(x)\in U_{j}\subset W$ for some $\tilde{h}\in\left \langle \mathscr{F} \right \rangle^{+}$. The converse is clear.

To obtain the second part of the proposition, we first observe that, arguing as before, we can define $V$ as a countable intersection of $V_i \cap W_i$, where
$V_{i}\eqdef\mathcal{O}_{\mathscr{F}}^+(U_{i})$.
For any  $x$ in such a (residual) dense set, we find that both the forward and backward $\mathscr{F}$-orbits of $x$ are dense in $X$. Conversely, it is clear that if $x$ has both forward and backward $\mathscr{F}$-orbits, then $\mathscr{F}$ is topologically transitive. Indeed, given nonempty open sets $U$, $V$ in $X$, there exist $g,f\in \langle \mathscr{F} \rangle^+$ such that $g(x)\in V$ and $f^{-1}(x)\in U$. Then, $h(U)\cap V\not=\emptyset$, where $h=g\circ f\in \langle \mathscr{F} \rangle^+$ and thus $\mathscr{F}$ is topologically transitive.
\end{proof}

\subsection{Robust transitivity.} \label{sec:rob-tran}

 Similar to \S\ref{sec:criterion_minimality}, we introduce the class of robustly transitive IFSs.

\begin{defi} \label{def:robtrans} We say that $\mathscr{F}$ in $\IFS_k(X)$ is \emph{robustly transitive} if there exists a neighborhood $\mathcal{U}$ of $\mathscr{F}$ in $\IFS_k(X)$ such that  any $\mathscr{G}\in \mathcal{U}$ is transitive.
\end{defi}

Here, we will extend Theorem~\ref{teo-expanding-minimal-implica-mini-robusta} for transitive IFSs.

\begin{teo}\label{teo:robust-trans}
Let $(X,\D)$ be a compact locally length metric space and consider  $\mathscr{F}\in \IFS_k(X)$. If $\mathscr{F}$ is  expanding and transitive, then for every $\varepsilon>0$ small enough, there is a neighborhood $\mathcal{V}$ of $\mathscr{F}$ in $\IFS_k(X)$ such that for any $\mathscr{G}\in \mathcal{V}$ and open ball $V$  of radius $\varepsilon$,
the $\mathscr{G}$-orbit $\OG(V)$ of $V$   is dense on $X$.
\end{teo}

The proof is basically a small modification of the proof of  Theorem~\ref{teo-expanding-minimal-implica-mini-robusta}. For this reason, we just indicate here the main difference.

\begin{proof}[Proof of Theorem~\ref{teo:robust-trans}]
Let $\mathscr{F} \in \IFS_k(X)$ be expanding and transitive. Let $\mathcal{U}$   and $\varepsilon>0$ be the neighborhood of $\mathscr{F}$ and any positive constant sufficiently small provided in the proof of Theorem~\ref{teo-expanding-minimal-implica-mini-robusta} from the expanding property. The main modification that we need to do is the following. Compare with Claim~\ref{claim:epsilon-minimal}.

\begin{afir} \label{claim:epsilon-transitive}
 There is a neighborhood
$\mathcal{V}$ of $\mathscr{F}$ in $\IFS_k(X)$ such that
any $\mathscr{G}\in\mathcal{V}$ the $\mathscr{G}$-orbit of any open ball $V$ of radius $\varepsilon$ is $\varepsilon$-dense, i.e.,  $\OG(V)\cap B\not=\emptyset$
for all  open ball $B$ of radius $\varepsilon$.
\end{afir}
\begin{proof}
By the compactness of $X$, we can cover it with finitely many open balls $W_1,\dots,W_n$ of radius $\varepsilon/4$. Since $\mathscr{F}$ is transitive, there are maps $h_{ij}\in \langle\mathscr{F}\rangle^+$, where $i,j=1,\dots,n$, such that $h_{ij}(W_j) \cap  W_i \not=\emptyset$. As this property only involves finitely many maps and open sets, we can find a neighborhood $\mathcal{V}$ of $\mathscr{F}$ in $\IFS_k(X)$ where it remains valid. This means that for any $\mathscr{G}\in \mathcal{V}$, the $\mathscr{G}$-orbit of $W_j$ visits $W_i$ for all $i,j=1,\dots,n$.  Since any ball with radius $\varepsilon$ contains at least one of the balls  ${W_1,\dots W_n }$ we get that $\OG(V)\cap B\not=\emptyset$
for all pair of  open balls $B$ and $V$ of radius $\varepsilon$. 
\end{proof}

Fix $\mathscr{G}\in \mathcal{W}=\mathcal{U}\cap \mathcal{V}$ and an open ball $V$ of radius $\varepsilon$. We will prove the density of the $\mathscr{G}$-orbit $\OG(V)$. 
To do this, according to Claim~\ref{claim:epsilon-transitive},  we have first that the orbit $\OG(V)$ is $\varepsilon$-dense.
Now, using the notation introduced in Theorem~\ref{teo-expanding-minimal-implica-mini-robusta},  we will show that actually
$\OG(V)$ is $\kappa^{-1}\varepsilon$-dense in $X$. To do this, take any $y\in X$ and consider the open ball $B(y,\kappa^{-1}\varepsilon)$.
Since $\kappa^{-1}\varepsilon<\delta$, then there is
$i\in\{1,2,\dots,n\}$ such that $B(y,\kappa^{-1}\varepsilon)\subset B_{i}$.
Then, by Claim~\ref{claim:deep} we have
\begin{equation*}
  h_{i}(B(h_i^{-1}(y),\kappa{\kappa}^{-1}\varepsilon))\subset B(y,\kappa^{-1}\varepsilon).
\end{equation*}
Since $\OG(V)$ is $\varepsilon$-dense, there is $h\in\langle\mathscr{G}\rangle^+$ and $x\in V$
such that $$h(x)\in B(h_i^{-1}(y),\varepsilon)=B(h_i^{-1}(y),\kappa{\kappa}^{-1}\varepsilon).$$
Therefore, $(h_{i}\circ h)(x)\in B(y,\kappa^{-1}\varepsilon)$. Consequently,
the orbit $\OG(V)$ is $\kappa^{-1}\varepsilon$-dense. Repeating this argument, we can show that the orbit $\OG(V)$
is $\kappa^{-p}\varepsilon$-dense for all $p\in\mathbb{N}$ and consequently,  $\OG(V)$ is dense as we wanted to show.
\end{proof}

Given $\mathscr{F}\in \IFS_k(X)$, we consider the \emph{inverse family} $\mathscr{F}^{-1}\in \IFS_k(X)$ as in Remark~\ref{rem:backward-trans}. In what follows, we say for brevity that a property $P$ of $\mathscr{F}$ holds \emph{forward and backward} if both $\mathscr{F}$ and $\mathscr{F}^{-1}$ satisfy $P$.
\begin{teo}
    Let $(X,\D)$ be a compact locally length space and consider $\mathscr{F}\in \IFS_k(X)$. If $\mathscr{F}$ is transitive, forward and backward expanding, and has a blending region, then $\mathscr{F}$ is robust transitive.
\end{teo}
\begin{proof} Let $B$ be the blending region of $\mathscr{F}$. Since $B$ is an open set, we have  $\delta>0$ and $x \in B$ such $B(x,\delta)\subset B$.
According to Remark~\ref{rem:backward-trans}, $\mathscr{F}$ is forward and backward transitive. Since it is also forward and backward expanding, we can apply Theorem~\ref{teo:robust-trans} for $\mathscr{F}$ and $\mathscr{F}^{-1}$. Then, for a fixed $0<\varepsilon<\delta$ small enough, we have a neighborhood $\mathcal{V}$ of $\mathscr{F}$ in $\IFS_k(X)$ such that for any $\mathscr{G}\in \mathcal{V}$ the forward and backward $\mathscr{G}$-orbits $\OG(W)$ and $\mathcal{O}_{\mathscr{G}}^-(W)$ of $W=B(x,\varepsilon)$ are both  dense on $X$. Then, given any pair of open sets $U$ and $V$, there exists $h,g\in\langle\mathscr{G}\rangle^+$ such that $h^{-1}(U)\cap W \not=\emptyset$ and $g(V)\cap W \not=\emptyset$.  Since $W\subset B$ and $B$ is a blending region, according to Lemma%
~\ref{lem-densidad-A} we find $H \in \langle\mathscr{G}\rangle^+$ such that $H(g(V)\cap W) \cap (h^{-1}(U)\cap W)\not=\emptyset$. In particular, we get that $f(V) \cap U \not=\emptyset$ where $f=h\circ H\circ g \in \langle\mathscr{G}\rangle^+$. This implies that $\mathscr{G}$ is transitive concluding the proof.
\end{proof}

We say that $\mathscr{F}\in \IFS_k(X)$ is  \emph{symmetric} if $f^{-1}\in \mathscr{F}$ provided
$f\in \mathscr{F}$. Notice that the semigroup $\langle \mathscr{F}\rangle^+$ generated by a symmetric $\mathscr{F} \in \IFS_k(X)$ is actually a group.
Moreover,  an expanding symmetric IFS also backward expanding. Thus, we get the following immediate corollary.

\begin{cor} \label{cor:tran-expanding-blending}
    Let $(X,\D)$ be a compact locally length  space and consider a symmetric $\mathscr{F}\in \IFS_k(X)$. If $\mathscr{F}$ is transitive,  expanding, and has a blending region, then $\mathscr{F}$ is robust transitive.
\end{cor}

\section{Minimality in robustly transitive circle IFS}\label{sec:minimal-in-rt-S1}

The goal of this section is to show Theorems~\ref{mainthm-IFS} and~\ref{mainteo:exemplos-IFS}. To do this, we need some preliminaries.
 Fix $1\leq r\leq \infty$ and $k\geq 1$.
 Recall the notations $\IFS_k^r(\mathbb{S}^1)$ and $\mathrm{RT}^r_k(\mathbb{S}^1)$ introduced in~\S\ref{sec:gen-min-intro} and~\S\ref{IFS-intro}. Observe that $\mathrm{RT}^r_k(\mathbb{S}^1)$ is an open set of $\IFS^r_k(\mathbb{S}^1)$.  Let $\mathscr{F} \in \IFS_k^r(\mathbb{S}^1)$ be a finite family of orientation-preserving circle $C^r$ diffeomorphisms.

\begin{defi}
A union $U$ of finitely many open intervals, $U\not\in\{\mathbb{S}^{1},\emptyset\}$, is called an \emph{interval-domain}. An interval-domain $U$ is \emph{strictly absorbing} for $\mathscr{F}$ if
$\overline{f(U)}\subset U$ for all $f\in \mathscr{F}$.
\end{defi}

The following result gives us a condition for the nonexistence of a strictly absorbing interval-domain.
\begin{prop} \label{prop:no-abs}
If $\mathscr{F}$ is transitive, then $\mathscr{F}$ has no strictly absorbing interval-domain.
\end{prop}
\begin{proof}
We proceed by contradiction. Suppose there is a strictly absorbing interval-domain $U$ for $\mathscr{F}$. Then, by the invariance of the absorbing interval-domain and since $U\not=\emptyset$, we have that  $\emptyset\not=\overline{\OF(U)}\subset U$. On the other hand, since $\mathscr{F}$
is transitive, we have $\mathbb{S}^{1}=\overline{\OF(U)}\subset U$. This contradicts the fact that $U$ is an interval-domain ($U\not = \mathbb{S}^1$) and conclude the proof.
\end{proof}

Note that the set of IFSs with a strictly absorbing interval-domain is open because such domains persist under small $C^{0}$ perturbations. Let us denote the $C^{r}$-interior of the complement of this set by $\mathcal{N}^{r}_k$. More precisely, $\mathcal{N}^{r}_k$ denotes the interior (with respect to the topology provided by the metric $\bar{\D}_r$) of the set of families $\mathscr{G} \in \IFS_k^r(\mathbb{S}^1)$ such that $\mathscr{G}$ has no strictly absorbing interval-domain.

\begin{obs}  \label{rem:robtran-N}
Note that if $\mathscr{F}$ is robustly transitive, then according to Proposition~\ref{prop:no-abs}, it has no strictly absorbing interval-domain in a robust manner. As a consequence, $\mathrm{RT}_{k}^{r}(\mathbb{S}^{1}) \subset \mathcal{N}^{r}_k$.
\end{obs}

The following result was showed by Kleptsyn, Kudryashov and Okunev in~\cite[Thm.~1.5, Cor.~5.7]{kleptsyn2018classification}.

\begin{teo}[{\cite{kleptsyn2018classification}}]\label{teo-1.5-Kleptsyn}
For any $k\geq 2$ and $1\leq r\leq \infty$, there exists a dense set $\mathcal{D}$ within $\mathcal{N}^{r}_k$ such that every $\mathscr{F} \in \mathcal{D}$ is both forward and backward robustly minimal. Additionally, the elements of $\mathscr{F}$ are Morse-Smale diffeomorphisms and there is a map $g\in \langle \mathscr{F}\rangle^+$ which is $C^\infty$ conjugated to an irrational rotation.

 Furthermore, any family $\mathscr{G}$ in $\IFS_k^r(\mathbb{S}^1)$ hat does not have a strictly absorbing interval-domain can be approximated by IFSs in $\mathcal{D}$.
 \end{teo}

The mentioned approximation in the second part of the above theorem follows from the proof, as obtaining irrational rotations only requires that the IFS does not have a strictly absorbing interval-domain (see the proof of \cite[Cor~5.7]{kleptsyn2018classification}). Moreover,  as an immediate consequence of Proposition~\ref{prop:no-abs} and Theorem~\ref{teo-1.5-Kleptsyn} we have the following:

\enlargethispage{0.3cm}

\begin{cor}\label{teo-aproxi-elem-D}
     Let $\mathscr{F}$ be a transitive family in $\IFS^r_k(\mathbb{S}^1)$, where $r\geq 1$ and $k\geq 2$. Then, there exists $\mathscr{G} \in \IFS^r_k(\mathbb{S}^1)$, $C^{r}$ close to $\mathscr{F}$, such that $\mathscr{G}$ is both robustly forward and backward minimal, the elements of $\mathscr{G}$ are Morse-Smale diffeomorphisms, and there is a map $g_0\in \langle \mathscr{G}\rangle^+$ that is $C^\infty$ conjugate to an irrational rotation.
\end{cor}

Now we are ready to prove Theorem~\ref{mainthm-IFS}.


\begin{proof}[Proof of Theorem~\ref{mainthm-IFS}]
Let $\mathcal{W}$ be the set of families $\mathscr{F}$ in $\IFS^r_k(\mathbb{S}^1)$ comprising Morse-Smale diffeomorphisms that are both forward and backward robustly minimal. Since Morse-Smale diffeomorphisms constitute an open set, and minimality implies transitivity (see Remark~\ref{rem:min-tran}), we conclude that $\mathcal{W}$ is an open set of $\mathrm{RT}^r_k(\mathbb{S}^1)$. Moreover, according to Remark~\ref{rem:robtran-N}, we have $\mathcal{D} \subset \mathcal{W} \subset \mathrm{RT}^r_k(\mathbb{S}^1) \subset \mathcal{N}^{r}_k$, where $\mathcal{D}$ is the dense set within $\mathcal{N}^{r}_k$ provided by Theorem~\ref{teo-1.5-Kleptsyn}. Therefore, $\mathcal{W}$ is open and dense in $\mathrm{RT}^r_k(\mathbb{S}^1)$. In fact, since any $f$ in  $\mathscr{F}\in \mathcal{W}$ is a Morse-Smale
diffeomorphism, any invariant probability measure of $f$ needs to be finitely supported (on the set of periodic orbits). This implies that $\mathscr{F}$ has no common invariant measure, since otherwise, a common invariant
measure $\mu$ satisfies
\begin{equation*}\label{eq-statio}
    \mu=\frac{1}{k}\sum_{i=1}^{k}{\mu\circ f_i^{-1}}.
\end{equation*}
i.e., it is a stationary measure of $\mathscr{F}$,
and then, by~\cite[Prop.~5]{kleptsyn2004contraction}, it needs to be fully supported on $\mathbb{S}^1$. Therefore, by ~\cite[Prop. 3.10]{BGMS17}, it concludes that $\mathbb{S}^1$
is a strict attractor of $\mathscr{F}$. Repeating the same argument for $\mathscr{F}^{-1}$, we also get that  $\mathbb{S}^1$ is a strict attractor of $\mathscr{F}^{-1}$.

Finally, the approximation of a transitive family in $\IFS_k^r(\mathbb{S}^1)$  by IFSs in $\mathcal{W}$ follows from Proposition~\ref{prop:no-abs} and the approximation in the second part of Theorem~\ref{teo-1.5-Kleptsyn}. This completes the proof.
\end{proof}

\begin{proof}[Proof of Corollary~\ref{cor:ergodicity}]
    Let $\mathscr{F}$ be a transitive  family in $\IFS^r_k(\mathbb{S}^1)$, where $r>1$ and $k\geq 2$. By Corollary~\ref{teo-aproxi-elem-D}, there exists $\mathscr{G} \in \IFS^r_k(\mathbb{S}^1)$ close to $\mathscr{F}$ 
    such that $\mathscr{G}$ is both forward and backward robustly minimal, the elements of $\mathscr{G}$ are Morse-Smale diffeomorphisms, and there is a map $g_0\in \langle \mathscr{G}\rangle^+$ that is $C^\infty$ conjugate to an irrational rotation. In particular, $\mathscr{G}\in \mathcal{W}$ as stated in Theorem~\ref{mainthm-IFS}.

    Let $\phi$ be the $C^\infty$ conjugacy map between $g_0$ and the rigid irrational rotation. Consider a neighborhood $\mathcal{U}_0$ of $\mathscr{G}$ in $\IFS^r_k(\mathbb{S}^1)$
such that every $\mathscr{H}\in \mathcal{U}_0$ is both forward and backward minimal. Set  $\hat{\mathcal{U}}_0=\{\phi\circ \mathscr{H} \circ \phi^{-1} \in \IFS^r_k(\mathbb{S}^1):  \mathscr{H} \in \mathcal{U}_0\}$ where
$\phi\circ \mathscr{H} \circ \phi^{-1} =\{
\phi\circ h \circ \phi^{-1}: h\in \mathscr{H}\}$. Notice that the semigroup $\langle \hat{\mathscr{G}}\rangle^+$ has a rigid irrational rotation where $\hat{\mathscr{G}}=\phi\circ \mathscr{G}\circ\phi^{-1}\in \hat{\mathcal{U}}_0$. Moreover, since the property of being a Morse-Smale diffeomorphism is invariant under smooth conjugation, $\langle \hat{\mathscr{G}}\rangle^+$ also contains a Morse-Smale diffeomorphism. Hence, by Theorem~\ref{teo-B-BFS14}, we conclude that $\hat{\mathscr{G}}$ is expanding and robustly minimal. Since the expanding property is also a robust property (see Remark~\ref{obs-expanding-es-robusta}), we get a subset $\mathcal{U}$ of $\mathcal{U}_0$ such that for any $\mathscr{H} \in \mathcal{U}$,  $\hat{\mathscr{H}}=\phi \circ \mathscr{H}\circ \phi^{-1}$ is expanding and minimal. By~\cite[Thm.~B]{BFMS16}, $\hat{\mathscr{H}}$ is ergodic with respect to the normalized Lebesgue measure $\mathrm{Leb}$ on $\mathbb{S}^1$.

To conclude the proof, let $A$ be an $\mathscr{H}$-invariant Borel set on $\mathbb{S}^1$, i.e., $h(A)\subset A$ for all $h\in \mathscr{H}$.  Since $(\phi\circ h\circ \phi^{-1})(\phi(A))\subset \phi(A)$, then $\phi(A)$ is $\hat{\mathscr{H}}$-invariant. Thus, $\mathrm{Leb}(\phi(A))\in \{0,1\}$. But since $\phi$ is smooth (at least $C^1$), we can conclude that $\mathrm{Leb}(A)\in \{0,1\}$, completing the proof of the corollary.
\end{proof}

\subsection{Zoo of examples}
In order to prove Theorem~\ref{mainteo:exemplos-IFS}, we utilize group actions of circle diffeomorphisms. In the following subsection, we present some preliminar results to prove this theorem,
which may also be of independent interest in this subject.

\subsubsection{Exceptional minimal sets.}

Consider a group $G$ of homeomorphisms of the circle. A symmetric finite family $\mathscr{F}$ of homeomorphisms of the circle is a \emph{generating system} for $G$ if $G=\left \langle \mathscr{F} \right \rangle^{+}$. By a \emph{$G$-orbit}, we understand the action of $G$ at a point $x \in \mathbb{S}^1$, that is, the set of points $G(x) = \{ g(x) : g \in G  \}$. A subset $B \subset \mathbb{S}^1$ is \emph{$G$-invariant} if $g(B) = B$ for all $g \in G$. We say that a {$G$-invariant} subset $A \subset \mathbb{S}^1$ is \emph{minimal} for $G$ if the $G$-orbit of $x$ is dense in $A$ for all $x \in A$. Recall that a set $K$ is a Cantor set if it is compact, totally disconnected, perfect, and uncountable. A set is totally disconnected if its connected components are singletons, and perfect if it is closed and has no isolated points. There are only three possible options for $G$~\cite{Navas,Gh01}:
\begin{enumerate}[itemsep=0.1cm]
\item there is a finite $G$-orbit,
\item every $G$-orbit is dense on the circle, or
\item there is a unique $G$-invariant minimal Cantor set.
\end{enumerate}
If $G(x)$ has finitely many different elements, it is called a \emph{finite orbit}, while if its closure is $\mathbb{S}^1$, it is called a \emph{dense orbit}. In the second conclusion, the group $G$ is said to be \emph{minimal}. In this case, every closed $G$-invariant set is either empty or coincides with the whole space $\mathbb{S}^1$. The Cantor set $K$ in the third conclusion is usually called an \emph{exceptional minimal set}. This Cantor set $K$ satisfies
$$
g(K)=K \ \ \text{for all $g\in G$} \quad  \text{and} \quad
K=\overline{G(x)} \ \
   \text{for all $x\in K$}.
$$
Notice that these properties are equivalent to saying that $K$ is minimal with respect to the inclusion of $G$-invariant closed sets.

There are two well-known examples of exceptional minimal sets: those obtained from Denjoy examples  and those obtained from Schottky groups. The \emph{Denjoy examples}  are defined by the action of a cyclic group of orientation-preserving circle diffeomorphisms with irrational rotation numbers. The diffeomorphisms in these examples are usually not smooth and are of class $C^{r}$ for $0\leq r<2$. On the other hand, \emph{Schottky groups} are non-abelian groups of orientation-preserving  circle  smooth diffeomorphisms.  They are generated by a finite collection of disjoint open arcs, called \emph{Schottky sets}, which are taken to other Schottky sets under the action of the group elements. For more details, see~\cite{Navas}.

Since all Cantor sets are homeomorphic (see~\cite[Theorem~8]{Moise1977}), any Cantor set can be an exceptional minimal set of a group action of circle homeomorphisms. However, the regularity of the group elements can limit certain Cantor sets from being the exceptional minimal set. For example, in~\cite{Mc81}, McDuff examines the \emph{gap ratio} required for a Cantor minimal set to be the exceptional minimal set of a Denjoy example. He concludes that the standard middle Cantor set cannot be the exceptional minimal set of any cyclic $C^1$ group. Nonetheless, this Cantor set can be preserved by non-cyclic smooth groups, as Ghys and Sergiescu showed in~\cite[proof of Theorem~2.3]{GS}. We will reproduce a similar construction to prove the following.



\begin{prop}\label{propo-existencia-minimal-excep}
For any $1\leq r\leq \infty$, there exists a $4$-generated group $G$ of $\mathrm{Diff}^r_+(\mathbb{S}^{1})$  such that
\begin{enumerate}
\item $G$ has an exceptional minimal set  $K$;
\item the $G$-orbit of every point of $\mathbb{S}^{1}\setminus K$ is dense in $\mathbb{S}^{1}$;
\item  there is $h \in \mathrm{Diff}_+^r(\mathbb{S}^{1})$ such that $h(K) \supsetneq K$.
\end{enumerate}
Additionally, if $\mathscr{G}_0$ is a symmetric generating system for $G$, then the following properties also holds:
\begin{enumerate}[resume]
    \item $\mathscr{G}_0$ is expanding;
    \item there is a blending region for $\mathscr{G}_0$.
\end{enumerate}
\end{prop}

Before proving the proposition, we need to introduce Thompson's groups. To do this, consider first $\mathbb{S}^1$ as the unit interval $[0,1]$ with the endpoints identified. Thompson's group $T$ is the group of orientation-preserving piecewise linear homeomorphisms of $\mathbb{S}^1$ that are differentiable except at finitely many dyadic rational numbers. On intervals of differentiability, the derivatives are powers of 2, and this group maps the set of dyadic rational numbers onto itself. This group is finitely presented and $2$-generated~\cite{BHS22}.

A remarkable property of Thompson's group $T$ is that it can be represented as a group of $C^\infty$ circle diffeomorphisms.
According to~\cite[Lemma 1.8]{GS}, for each $C^r$ diffeomorphism $H : \mathbb{R} \to \mathbb{R}$ satisfying the following properties:
\begin{enumerate}
\item[(i)] for each $x \in \mathbb{R}$ one has
$H(x + 1) = H(x) + 2$;
\item [(ii)] $H(0)=0$;
\item [(iii)] $H'(0)= 1$ and $H^{(i)}(0)=0$ for all $i= 2, \dots, r$.
\end{enumerate}
There is, a representation of $T$ in $\mathrm{Diff}^r_+(\mathbb{S}^1)$, i.e., a homomorphism $\Phi_H : T \to \mathrm{Diff}^r_+(\mathbb{S}^1)$.

If the diffeomorphism $H : \mathbb{R} \to \mathbb{R}$ also  has at least two fixed points, then the group $\Phi_H(T)$ has an exceptional
minimal set,  as stated in~\cite[Prop.~1.14]{GS}.
This result is a consequence of the fact that the dynamics of the induced map $\bar{H}$ on the circle by $H$ determines the orbits of the group $\Phi_H(T)$. Namely, consider the equivalence relation defined on $\mathbb{S}^1$, where $x,y\in \mathbb{S}^1$ are considered related if and only if there exists a non-negative integer $n$ such that $y=\bar{H}^n(x)$ or vice-versa. The equivalence classes resulting from this relation are shown to be equivalent to the orbits of the group $\Phi_H(T)$. See Lemma 1.11 in \cite{GS}. The authors provide a proof of this lemma by constructing a sequence of maps $g_n$ of $T$ and arcs $[x_n,x_{n+1}]$ in $\mathbb{S}^1$, whose union is $\mathbb{S}^1$, such that the restriction of $\Phi_H(g_n)$ to $[x_n,x_{n+1}]$ coincides with $\bar{H}$.

\begin{proof}[Proof of Proposition~\ref{propo-existencia-minimal-excep}]
To prove the proposition, we will consider the following function~$H$. We first consider a $C^\infty$ function $\bar{H}$ as shown in Figure~\ref{fig}, defined in the interval $P=[0,\frac{1}{3}] \cup [\frac{5}{12},1]$, which satisfies conditions (ii) and (iii) for $r=\infty$. We extend $\bar{H}$ to $[\frac{1}{3},\frac{5}{12}]$ using the Whitney extension theorem~\cite{CS18,CS19}\footnote{In the case that $P$ is a finite  union of disjoint compact intervals, Whitney extension theorem asserts that a function $f:P \to \mathbb{R}$ admits an extension of class $C^r$ to the convex hull of $P$ if, and only if, $f$ is the class $C^r$ on $P$.} applied to $H(x) = \bar{H}(x)$ for $x\in [0,\frac{1}{3}]$ and $H(x) = \bar{H}(x) + 1$ for $x\in [\frac{5}{12},1]$. Finally, we define $H(x)$ for $x\not \in [0,1]$ such that it satisfies condition (i).

\begin{figure}
  \begin{center}
    \begin{tikzpicture}
\begin{axis}[
  ticklabel style={font=\small},
  axis lines=box,
  xmin=0, xmax=1,
  ymin=0, ymax=1,
  xtick={0,0.25,0.33,0.41,0.5,1},
  xticklabels={0,$\frac{1}{4}$,$\frac{1}{3}$,
  $\frac{5}{12}$,$\frac{1}{2}$,$1$},
  ytick={0.25,0.5,1},
  yticklabels={$\frac{1}{4}$,$\frac{1}{2}$,$1$},
  width=11cm, height=11cm,
  grid=major, 
    grid style={gray!30} 
  ]
\addplot[
    domain=0:1,
    samples=100,
  ] {x};
  \addplot[
    domain=0.25:0.333333,
    samples=100,
    color=blue,
  ] {3*(x-1/4)+1/4};
  \addplot[
    domain=0.41666666666666:0.5,
    samples=100,
    color=blue,
  ] {3*(x-1/2)+1/2};
  \addplot[
    domain=0:0.25,
    samples=100,
    color=blue,
  ] {x-8*x^2+32*x^3};
  \addplot[
    domain=0.5:1,
    samples=100,
    color=blue,
  ] {8*x^3-20*x^2+17*x-4};
\end{axis}
\end{tikzpicture}
     \caption{Graph of the function $\bar{H}$ on $P=[0,\frac{1}{3}] \cup [\frac{5}{12},1]$. In particular,  the function takes the form $\bar{H}(x)=3(x-\frac{1}{4})+\frac{1}{4}$ on the interval $[\frac{1}{4}, \frac{1}{3}]$, and $\bar{H}(x)=3(x-\frac{1}{2})+\frac{1}{2}$ on the interval $[\frac{5}{12}, \frac{1}{2}]$.}
     \label{fig}
\end{center}
\end{figure}
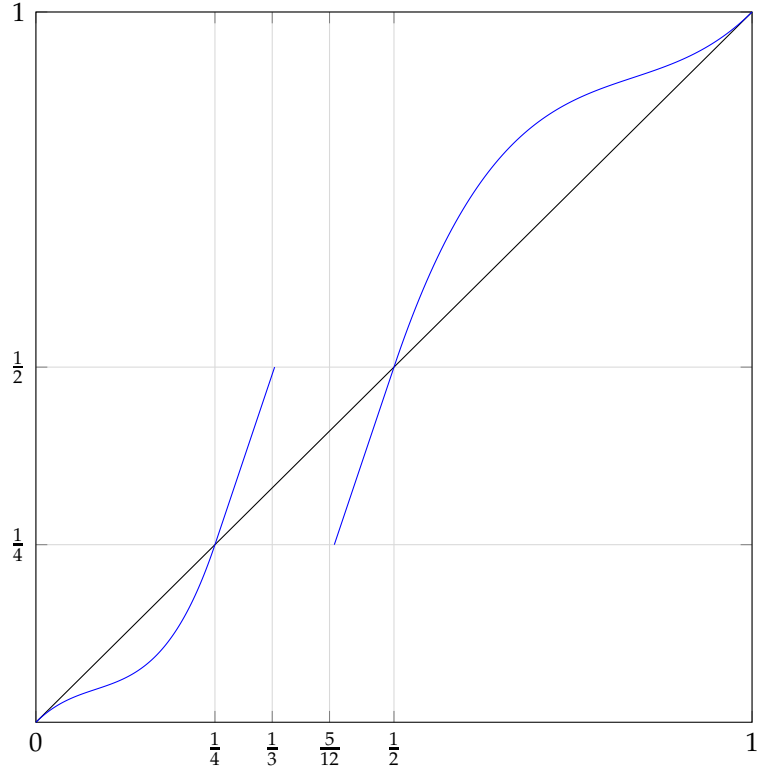


We observe that $$K_I = \bigcap_{n \geq 0} \bar{H}^{-n}(I),$$
where $I = [\frac{1}{4},\frac{1}{2}]$, is the standard middle Cantor set in $I$. The sets $\bar{H}^{-n}(\mathbb{S}^1 \setminus I)$ for $n \geq 1$ are the gaps of this Cantor set. Thus, $K = K_I$ is the exceptional minimal set of the group $\Phi_H(T)$, and there is only one class of gaps, meaning that for all gaps $V$ and $U$, there exists a $g \in \Phi_H(T)$ such that $g(U) = V$.

We also have that $h(K_I) = K_J$, where $h(x) = 3(x - \frac{1}{4}) + \frac{1}{4}$ for $x \in I$ and $K_J$ is the standard middle Cantor set in $J = [\frac{1}{4}, 1]$. See~\cite[Theorem B]{BMPV97} where the existence of a smooth function between central Cantor sets is characterized. Since $K_I \subset K_J$, we get that $h(K_I) \supsetneq K_I$. Using the Whitney Extension Theorem as above, we obtain an orientation-preserving $C^\infty$ diffeomorphism of $\mathbb{S}^1$ satisfying the property 3 above with $K = K_I$.

Now, we will demonstrate the density of the orbits of $\mathbb{S}^1\setminus K$ by adding two extra generators to $\Phi_H(T)$. First, fix a gap $U$ of $K$. Let $f$ and $g$ be as in Figure~\ref{fig2}.

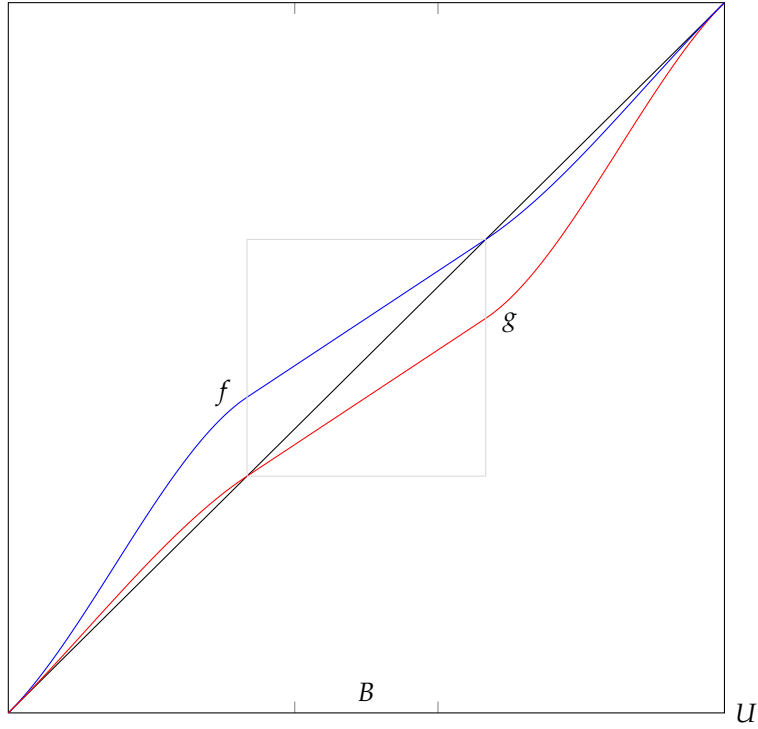
\begin{figure}
  \begin{center}
    \begin{tikzpicture}
\begin{axis}[
  ytick style={draw=none}, 
  axis lines=box,
  xmin=0, xmax=1,
  ymin=0, ymax=1,
  width=11cm, height=11cm,
  xtick={0.4,0.6}, xticklabels={},
  ytick={}, yticklabels={},
  ]
\addplot[
    domain=0:1,
    samples=100,
  ] {x};
  \addplot[
    domain=0.333333:0.666666,
    samples=100,
    color=red,
  ] {(2/3)*(x-1/3)+1/3};
  \addplot[
    domain=0.333333:0.666666,
    samples=100,
    color=blue,
  ] {(2/3)*(x-2/3)+2/3};
  \addplot[
    domain=0:0.33333,
    samples=100,
    color=blue,
  ] {x+4*x^2-9*x^3};
  \addplot[
    domain=0.666666:1,
    samples=100,
    color=red,
  ] {-9*x^3+23*x^2-18*x+5};
    \addplot[
    domain=0:0.33333,
    samples=100,
    color=red,
  ] {-3*x^3+x^2+x};
  \addplot[
    domain=0.666666:1,
    samples=100,
    color=blue,
  ] {-3*x^3+8*x^2-6*x+2};
  \addplot[domain=0.3333:0.6666, samples=2,color=gray!30] {0.3333333};
  \addplot[domain=0.3333:0.6666, samples=2,color=gray!30] {0.66666};

\draw [gray!30] (0.33333,0.33333) -- (0.33333,0.66666);

\draw [gray!30] (0.66666,0.33333) -- (0.66666,0.66666);

\node at (0.7, 0.55)   {\small $g$};

\node at (0.3, 0.45)   {\small $f$};
\node at (0.5, 0.03)   {\small $B$};
\end{axis}
\node at (9.7, 0)   {\small $U$};
\end{tikzpicture}
     \caption{The figure depicts the maps $f$ and $g$, defined on a gap $U$ of the Cantor set $K$, and highlights the central interval $B$ as a blending region for $\{f,g\}$. \vspace{-0.5cm}
}
    \label{fig2}
\end{center}
\end{figure}

These maps can be extended to $\mathbb{S}^1$ by defining its restriction to $\mathbb{S}^1\setminus U$ as the identity map. Let us consider the group $G$ generated by $\Phi_H(T)$, $f$, and $g$. We note that $K$ is the exceptional minimal set for $G$. Let $\mathscr{G}_0$ be a symmetric generating family for $G$. Then, $\mathscr{G}_0$ has a blending region $B$ in $U$, and for every $x\in U$, there exists $S\in G=\langle \mathscr{G}_0\rangle^+$ such that $S^{-1}(x)\in B$.
Since $G$ has only one class of gaps of $K$, this implies that $\mathbb{S}^1\setminus K = G(B) = \mathcal{O}^+_{\mathscr{G}_0}(B)=\mathcal{O}^-_{\mathscr{G}_0}(B)$. Hence, according to Theorem~\ref{thm-BFMS16}, $\mathscr{G}_0$ is minimal in $\mathbb{S}^1\setminus K$. Given that $K$ is a Cantor set, it follows that the $G$-orbits of the points in $\mathbb{S}^1\setminus K$ are dense in $\mathbb{S}^1$ .

To complete the proof of the proposition, we need to establish that $\mathscr{G}_0$ is expanding. Recall that $K\subset L\eqdef [\frac{1}{4},\frac{1}{3}]\cup [\frac{5}{12},\frac{1}{2}]$, and observe that in $L$, $\bar{H}'=3$ and hence $\bar{H}'|_{K}>1$. As previously mentioned, there exist maps $g_n$ in $T$ and arcs $[x_n,x_{n+1}]$ in $\mathbb{S}^1$ whose union cover $\mathbb{S}^1$ such that $\Phi_H(g_n)|{[x_n,x_{n+1}]}=\bar{H}$. Thus, for any $x\in K$, we can find an arc $[x_n,x_{n+1}]$ and $g_n \in T$ such that $\Phi_H(g_n)'(x)>1$. Since $\Phi_H(T)\subset G=\langle \mathscr{G}_0 \rangle^+$, it follows that $\mathscr{G}_0$ is expanding in $K$. For $x\in \mathbb{S}^1\setminus K$, since the $G$-orbit of $x$ is dense, we can find $S\in G$ such that $S(x)$ belongs to the blending region $B$. By applying $f^{-1}$ or $g^{-1}$ as necessary, we obtain a map $R\in G$ such that $R(S(x)) \in B$ and $R'(S(x))$ can be made arbitrarily large. Hence, $\mathscr{G}_0$ is also expanding in $\mathbb{S}^1\setminus K$. Therefore, we have shown that $\mathscr{G}_0$ is expanding in $\mathbb{S}^1$, completing the proof of the proposition.
\end{proof}

\vspace{-0.2cm}
As a consequence of Proposition~\ref{propo-existencia-minimal-excep}, we get  $\mathscr{G}$ in Theorem~~\ref{mainteo:exemplos-IFS}.
\vspace{-0.2cm}

\begin{cor} \label{cor:rt} Let $\mathscr{G}_0$ be a symmetric generating system of the group $G$ in Proposition~\ref{propo-existencia-minimal-excep}. Then, $\mathscr{G}_0$ is robustly transitive. Moreover, for every $r\geq 1$ and $k\geq 8$, $\mathscr{G}_0$ can be extended to a family $\mathscr{G}\in \mathrm{RT}^r_k(\mathbb{S}^1)$ that is neither forward nor backward minimal.
\end{cor}

\begin{proof}
From Proposition~\ref{propo-existencia-minimal-excep}, $\mathscr{G}_0$ is expanding and has a blending region. Moreover, $\mathscr{G}_0$ is transitive since any open set $U$ intersects $\mathbb{S}^1\setminus K$ and the forward $\mathscr{G}_0$-orbits (the $G$-orbits) of these points are dense in $\mathbb{S}^1$. Thus, from Corollary~\ref{cor:tran-expanding-blending}, $\mathscr{G}_0$ is robustly transitive. To conclude the proof, we note that since $G$ is 4-generated, we can choose $\mathscr{G}_0$ to consist of 8 maps. By adding the identity map as many times as necessary, we obtain a family $\mathscr{G}$ in $\mathrm{RT}^r_k(\mathbb{S}^1)$ with $k\geq 8$ that is neither forward nor backward minimal.
\end{proof}

As noted in~\cite[Sec.~1.5.2]{Navas}, the dyadic feature in the Ghys-Sergiescu’s smooth realization is not essential. A similar construction can be performed for any integer $n \geq 2$ by using Thompson's group $T_{n,1}$ \cite{B87} and starting with a map $H$ satisfying $H(x + 1) = H(x) + n$ for all $x \in \mathbb{R}$. Thus, the argument in the proof of the above proposition can be generalized to show the following result. First, given $r \geq 1$ and a compact interval $J$, a Cantor set $K$ is said to be  \emph{preserving-orientation regular} in $J$ of class $C^r$  if there exist pairwise disjoint compact intervals  $I_1,\dots, I_k$, $k \geq  2$  where $J$ is the convex hull of the union $I \equiv I_1 \cup \dots \cup I_k$ and a $C^r$  functions $S : I \to J$ such that $S|_{I_j}$ is increasing,\footnote{In literature it is often requested that $S$ be expanding, meaning $|S'| > 1$. Here this condition is not required, although it is required that $S$ be monotonically increasing in order to preserve orientation.}
$S(I_j) = J$ for $j = 1, \dots, k$ and  $K = \bigcap_{
n\in\mathbb{N}} S^{-n}(I)$.

\begin{prop} \label{pro:Cantor} Let $J$ be a proper closed arc of $\mathbb{S}^1$ and consider any preserving orientation regular Cantor set $K$ in $J$ of class $C^r$ with $r\geq 1$ or $r=\infty$. Then,  there exists a  finitely generated group $G$ of $\mathrm{Diff}^r_+(\mathbb{S}^1)$ such that $K$ is the exceptional minimal set of $G$.
\end{prop}

\begin{obs} 
Bowen constructed in~\cite{bowen1975horseshoe} a regular Cantor set $L$ of class $C^1$ with positive Lebesgue measure. Then, according to Proposition~\ref{pro:Cantor}, it follows that $L$ can be an exceptional minimal set of a group in $\mathrm{Diff}^1_+(\mathbb{S}^1)$. Ghys and Sullivan posed a question regarding the Lebesgue measure of an exceptional minimal set of a finitely-generated group of $C^2$ circle diffeomorphisms, asking whether this measure must necessarily be equal to zero. A significant step towards resolving this question affirmatively was achieved by Deroin, Kleptsyn, and Navas in~\cite{DKN} for groups of real-analytic diffeomorphisms.
\end{obs}

Finally, we complete the proof of Theorem~\ref{mainteo:exemplos-IFS} constructing $\mathscr{F}$.

\begin{lem}\label{lema-previo-exem-forward-but-backward-minimal}
Let $G$ and $K$ be as in Proposition~\ref{propo-existencia-minimal-excep}. Then, the closed subsets of $\mathbb{S}^{1}$ which
are invariant by $G$ are $\emptyset,K$ and $\mathbb{S}^{1}$.
\end{lem}
\begin{proof}
Let $B$ be a closed subset of $\mathbb{S}^1$ that is invariant under the action of $G$. If $B\neq\emptyset$, then it follows from the minimality of $K$ that $K\subseteq B$. On the other hand, if $B\neq K$, then there exists a point $x\in\mathbb{S}^1\setminus K$ that belongs to $B$. By the invariance of $B$, the orbit of $x$ under the action of $G$ is also contained in $B$. However, the orbit of $x$ is dense in $\mathbb{S}^1$, which means that $B=\mathbb{S}^1$.
\end{proof}

\begin{prop} \label{prop:forward-not-backward}
Let $\mathscr{G}_0$ and $h$ be as defined in Proposition~\ref{propo-existencia-minimal-excep}. Then, the family $\mathscr{F}_0 = \mathscr{G}_0 \cup \{h\}$ is robustly transitive, not backward minimal and has $\mathbb{S}^1$ as a strict attractor. Additionally, for any $1 \le r \le \infty$ and $k \ge 8$, $\mathscr{F}_0$ can be extended to a family $\mathscr{F} \in \mathrm{RT}^r_{k+1}(\mathbb{S}^1)$ that is also not backward minimal and has $\mathbb{S}^1$ as a strict attractor, making it in particular forward minimal.
\end{prop}
\begin{proof}
Let $K$ be the exceptional minimal set of $G=\langle \mathscr{G}_0\rangle^+$, and let $K_1 = h(K)$. According to Proposition~\ref{propo-existencia-minimal-excep}, $K\subsetneq K_1$. The proof is divided into two parts:
\begin{enumerate}[label=\alph*)]
    \item $\mathscr{F}_0$ is not backward minimal: Since $K$ is the exceptional minimal set for $G$ and $\mathscr{G}_0$ is a symmetric generating system of $G$, we have that $g^{-1}(K)=K$ for all $g\in \mathscr{G}_0$. In addition, we have $h^{-1}(K)\subset h^{-1}(K_{1})=K$. Therefore, $K$ is forward invariant by $\mathscr{F}_0^{-1}$, and so $\mathscr{F}_0$ is not backward minimal.
    \item $\mathbb{S}^1$ is a strict attractor of $\mathscr{F}_0$: Let $B$ be a closed subset of $\mathbb{S}^1$ that is forward invariant by $\mathscr{F}_0$. Since $\mathscr{G}_0 \subset \mathscr{F}_0$ generates the group $G$, $B$ is also invariant by $G$. By Lemma~\ref{lema-previo-exem-forward-but-backward-minimal}, $B\in\{\emptyset,K,\mathbb{S}^{1}\}$. Moreover, $B\neq K$ because $K$ is not invariant by $h$ (otherwise, $K_{1}=h(K)=h(B)\subset B=K$, but $K_{1}$ strictly contains $K$). Hence, $B\in\{\emptyset,\mathbb{S}^{1}\}$, which means that $\mathscr{F}_0$ is forward minimal. Since $h$ contains an attracting fixed point, according to Proposition~\ref{Sarizadeh}, it follows that $\mathbb{S}^1$ is a strict attractor of $\mathscr{F}_0$.
\end{enumerate}

To complete the proof, we note that since $G$ is 4-generated, we can choose $\mathscr{G}_0$ to consist of 8 maps, and thus $\mathscr{F}_0$ has 9 elements. Additionally, according to Corollary~\ref{cor:rt}, $\mathscr{G}_0$ (and hence $\mathscr{F}_0$) is robustly transitive. By adding the identity map as many times as necessary, we can obtain a family $\mathscr{F}$ in $\mathrm{RT}^r_{k+1}(\mathbb{S}^1)$ with $k\geq 8$ with the desired properties.
\end{proof}

\section{Symbolic skew-product} \label{sec:skew-prodcut}
In this section, we will prove Theorems~\ref{mainthm1},~\ref{mainthm2} and~\ref{mainthm-examples}. To do so, we will begin by introducing the skew-product $\Phi=\tau\ltimes \mathscr{F}$, where $\mathscr{F}=\{f_1,\dots,f_k\}$ is a finite family of continuous maps of $X$, as previously discussed in~\S\ref{sec:minimality-foliations-intro}.

Given $\omega=(\omega_{i})_{i\in\mathbb{Z}}\in\Sigma_{k}=\{1,\dots,k\}^\mathbb{Z}$, we denote the sequence of compositions of map in $\mathscr{F}$ by
\begin{equation*}
  f^0_\omega=\mathrm{id}, \quad \text{and} \quad   f_{\omega}^{n}\eqdef f_{\omega_{n-1}}\circ\dots\circ f_{\omega_{0}}\quad\text{for all}\quad n>0.
\end{equation*}
Hence, with this notation we have that  \begin{equation*}
    \Phi^{n}(\omega,x)=(\tau^{n}(\omega),f_{\omega}^{n}(x))
    \quad   \text{for } \  \  (\omega,x)\in\Sigma_{k}\times X.
\end{equation*}
Given finite words $u=u_{-s}\dots u_{-1}$ and $v=v_{0}\dots v_{t}$ in the alphabet $\{1,\dots,k\}$, we define a cylinder in the bi-sequence space $\Sigma_{k}$ by
\begin{equation*}
  \begin{aligned}
  \left[u;v\right] &\eqdef\{\omega\in\Sigma_{k}: \omega_{i}=u_{i}, \ i=-s,\dots,-1, \ \ \text{and} \ \ \omega_{i}=v_{i}, \ i=0,\dots,t\},  \\
    \left[u\right]^{-}&\eqdef\{\omega\in\Sigma_{k}: \omega_{i}=u_{i}, \ i=-s,\dots,-1\},\\
    \left[v\right]^{+}&\eqdef\{\omega\in\Sigma_{k}:  \omega_{i}=v_{i}, \ i=0,\dots,t\}.
\end{aligned}
\end{equation*}
Note that the cylinder sets are open sets and form a basis for the topology of $\Sigma_k$.

\subsection{Transitivity.}

Similar to Definition~\ref{definicion-transitividad}, we define the skew-product $\Phi$ as \emph{transitive} if the (forward) $\Phi$-orbit  of every open set is dense in $\Sigma_{k}\times X$. Since Theorem~\ref{teo-equiva-transitividad} can be applied to the IFS comprising only $\Phi$, we find that the introduced notion of transitivity for $\Phi$ is equivalent to the notion of \emph{topological transitivity}. Specifically, this implies that for any pair of nonempty open sets $U$ and $V$ in $X$, there exists an $n\in\mathbb{N}$ such that $\Phi^n(U)\cap V \neq \emptyset$. Moreover, by Proposition~\ref{propo-equivalencia-transitividad-densidad-puntual}, since $X$ is a compact metric space (and thus a second countable Baire space), if $\mathscr{F}$ consists of homeomorphisms, then $\Phi$ is transitive if and only if it has a (forward) \emph{dense orbit}.\footnote{Proposition~\ref{propo-equivalencia-transitividad-densidad-puntual} implies the existence of a forward dense orbit from the transitivity. Conversely, for $\mathbb{N}$-actions, the existence of a forward dense orbit clearly implies topological transitivity (and hence transitivity). This is not true for a general semigroup action generated by more than one map.}
The next result shows that if $\mathscr{F}$ consists of open maps, then the transitivity of $\Phi=\tau\ltimes \mathscr{F}$ is also equivalent to the transitivity of the family $\mathscr{F}$ introduced in Definition~\ref{definicion-transitividad}.

\begin{teo} \label{teo:skew-transitive}
Let $\mathscr{F}$ be a finite family of continuous open maps of $X$.  Then $\mathscr{F}$ is transitive if and only if the symbolic skew-product $\Phi=\tau\ltimes \mathscr{F}$ is transitive.
\end{teo}

 \begin{proof}
By Theorem~\ref{teo-equiva-transitividad}, we know that $\mathscr{F}$ is transitive~\ref{item1-teo-equiv-transitividad} if and only if it is topologically transitive~\ref{item2-teo-equiv-transitividad}. Our goal now is to show the equivalence between~\ref{item2-teo-equiv-transitividad} and the topological transitivity of $\Phi$.

\begin{afir}\label{afirmacion-teo-equiv-transitividad}
The condition~\ref{item2-teo-equiv-transitividad} is equivalent to the following property: for every pair of non-empty open sets $U,V$ in $X$, there exists $n\in\mathbb{N}$ such that $\Phi^{n}(\Sigma_{k}\times U)\cap(\Sigma_{k}\times V)\neq\emptyset$.
\end{afir}
\begin{proof}
Suppose $\mathscr{F}$ satisfies property~\ref{item2-teo-equiv-transitividad} and take nonempty open sets $U,V$ in $X$. Then there exists $g\in\langle\mathscr{F}\rangle^{+}$ such that $g(U)\cap V\neq\emptyset$. Note that $g$ is of the form $g=f_{\omega'}^{n}$ for some $\omega'\in\Sigma_{k}$ and $n\in\mathbb{N}$. Thus, $f_{\omega'}^{n}(U)\cap V\neq\emptyset$. Therefore,
\begin{align*}
  \Phi^{n}(\Sigma_{k}\times U)\cap(\Sigma_{k}\times V)&=\Phi^{n}(\Sigma_{k}\times U)\cap (\tau^{n}(\Sigma_{k})\times V)\\
  &\supset\left(\{\tau^{n}(\omega')\}\times f_{\omega'}^{n}(U)\right)\cap \left(\{\tau^{n}(\omega')\}\times V\right)\neq\emptyset.
 \end{align*}

Conversely, assume that $\Phi$ is topologically transitive. Then, given nonempty open sets $U,V$ in $X$, there is $n\in\mathbb{N}$ such that
$$\emptyset \not=\Phi^{n}(\Sigma_{k}\times U)\cap(\Sigma_{k}\times V)=
 (\bigcup_{\omega\in\Sigma_{k}}{\{\tau^{n}(\omega)\}\times f_{\omega}^{n}(U)})\cap(\Sigma_{k}\times V)\neq\emptyset.
$$
Therefore, there exists $\omega\in \Sigma_k$ such that $f_{\omega}^{n}(U)\cap V\neq\emptyset$. Since $f^n_\omega \in \langle \mathscr{F}\rangle^+$, we have that $\mathscr{F}$ satisfies property~\ref{item2-teo-equiv-transitividad}.
\end{proof}
Now, if we assume that $\Phi$ is topologically transitive, there exists $n\in\mathbb{N}$ such that $\Phi^{n}(\Sigma_{k}\times U)\cap(\Sigma_{k}\times V)\neq\emptyset$ for all nonempty open sets $U,V$ in $X$. Then, by claim~\ref{afirmacion-teo-equiv-transitividad}, we have that $\mathscr{F}$ satisfies~\ref{item2-teo-equiv-transitividad}. Thus, to conclude the proof, it suffices to show that~\ref{item2-teo-equiv-transitividad} implies that $\Phi$ is topologically transitive.

To prove this, suppose~\ref{item2-teo-equiv-transitividad}. Consider nonempty open sets $U$ and $V$ in $X$, and finite words
$$\omega_{-}=\omega_{-s}\dots\omega_{-1},\quad\omega_{+}=\omega_{0}\dots\omega_{t},
\quad\omega'_{-}=\omega'_{-m}\dots\omega'_{-1},\quad\omega'_{+}=\omega'_{0}\dots\omega'_{r}$$
in the alphabet $\{1,\dots,k\}$, and take $C=[\omega_{-};\omega_{+}]$ and $D=[\omega'_{-};\omega'_{+}]$ as cylinders in $\Sigma_{k}$. Set $f=f_{\omega_{t}}\circ\dots\circ f_{\omega_{0}}$ and $g=f_{\omega'_{-1}}\circ\dots\circ f_{\omega'_{-m}}$. Since $\mathscr{F}$ consists of continuous open maps, we have that $f$ and $g$ are also continuous open maps. Then, $f(U)$ and $g^{-1}(V)$ are open sets, and since $\mathscr{F}$ satisfies~\ref{item2-teo-equiv-transitividad}, there exists $h=f_{\bar{\omega}_{l}}\circ\dots\circ f_{\bar{\omega}_{1}}\in\langle\mathscr{F}\rangle^{+}$ such that $h(f(U))\cap g^{-1}(V)\neq\emptyset$. Hence, $g\circ h\circ f(U)\cap V\neq\emptyset$.

Consider the cylinder $A=[\omega_-;\omega_+\bar{\omega}_1\dots\bar{\omega}_l\omega'_-\omega'_+]$.
Let $\omega$ be any bi-sequences in $A$. Then, $\omega\in C$ and making $n=t+l+m+1$, we have $\tau^{n}(\omega)\in D$
and $f_{\omega}^{n}=g\circ h\circ f$. Therefore,
\begin{equation*}
    \begin{aligned}
    \Phi^{n}(C\times U)\cap(D\times V)&\supset \Phi^{n}(\{\omega\}\times U)\cap(D\times V)\\
    &=(\{\tau^{n}(\omega)\}\times f_{\omega}^{n}(U))\cap(D\times V)\neq\emptyset
    \end{aligned}
\end{equation*}
and we conclude that $\Phi$ is topologically transitive.
\end{proof}

\subsection{Minimality of Strong foliation}
In this subsection, we recall the notion of \emph{strong foliations} for symbolic one-step skew-product maps following~\cite{NP12, HN13}. We will study the conditions under which these foliations are minimal.

Throughout the remainder of this section, we assume that $X$ is a compact metric space, and we consider the family $\mathscr{F}=\{f_{1},\dots,f_{k}\}$ consisting of homeomorphisms on $X$. In this case, $\Phi=\tau\ltimes \mathscr{F}$ is an invertible map, and we have \begin{align*}
\Phi^{-n}(\omega,x)&=(\tau^{-n}(\omega),f_{\omega}^{-n}(x))
\quad\text{where} \ \ \ f_{\omega}^{-n} \eqdef f_{\omega_{-n}}^{-1}\circ\dots\circ f_{\omega_{-1}}^{-1} \ \ \text{for all} \ \ n>0.
\end{align*}
\begin{obs} \label{rem:conj}
Notice that  $\Phi^{-1}$ is conjugate with the symbolic skew-product $\Psi=\tau\ltimes \mathscr{F}^{-1}$. Namely, $\Phi^{-1}=I^{-1}\circ \Psi \circ I$, where $I$ is the involution  $((\omega_i)_{i\in \mathbb{Z}}, x) \mapsto ((\omega'_i)_{i\in\mathbb{Z}},x)$ with $\omega'_i=\omega_{-i-1}$ for all $i\in\mathbb{Z}$.
\end{obs}
As introduced in~\S\ref{sec:minimality-foliations-intro}, the local strong stable and unstable sets of $(\omega,x)\in\Sigma_{k}\times X$ are given by
$W_{loc}^{ss}(\omega,x )=W_{loc}^{s}(\omega )\times\{x\}$ and $ W_{loc}^{uu}(\omega,x )=W_{loc}^{u}(\omega )\times\{x\}$ respectively.
Notice that
$$
 \Phi(W_{loc}^{ss}(\omega,x ))\subset W_{loc}^{ss}(\Phi(\omega,x) ) \quad \text{and} \quad
 \Phi^{-1}(W_{loc}^{ss}(\omega,x ))\subset W_{loc}^{ss}(\Phi^{-1}(\omega,x) ).
$$
That is, $\Phi$ sends local stable sets to local stable sets. Moreover, since
$$
d(\Phi^n(\omega,x),\Phi^n(\omega',x')) \leq \nu^n \D_{\Sigma_k}(\omega,\omega') \quad   \text{for all } (\omega',x')\in W^{ss}_{loc}(\omega,x ) \ \ \text{and} \ \ n>0
$$
the points in $W_{loc}^{ss}(\omega,x )$ are exponentially contracted. Similarly, $\Phi^{-1}$ sends local unstable sets to local unstable sets by exponentially contracting their points.

On the other hand, the global strong stable  and unstable leaves $W^{ss}(\omega,x)$ and $W^{uu}(\omega,x)$ of $(\omega,x)\in\Sigma_{k}\times X$ are defined by iterating the local ones as in Equation~\eqref{eq:ss-uu-global}.
Due to the invariance of the local strong stable leaf, we have
$$
 \Phi(W^{ss}(\omega,x ))\subset W^{ss}(\Phi(\omega,x) ) \quad \text{and} \quad
 \Phi^{-1}(W^{ss}(\omega,x ))\subset W^{ss}(\Phi^{-1}(\omega,x) ).
$$
Furthermore, as the local strong leaves are exponentially contracted, we have
\begin{equation*}
    \begin{aligned}
       W^{ss}(\omega,x )& \subset W^s(\omega,x ) \eqdef \{(\omega',x')\in \Sigma_k\times X: \lim_{n\to \infty}d(\Phi^n(\omega,x),\Phi^n(\omega',x'))=0\} \\
       W^{uu}(\omega,x )&\subset W^u((\omega,x) )\eqdef\{(\omega',x')\in \Sigma_k\times X: \lim_{n\to \infty}d(\Phi^{-n}(\omega,x),\Phi^{-n}(\omega',x'))=0
       \}.
    \end{aligned}
\end{equation*}
\begin{obs} \label{obs:duality}
Using the duality of the above definitions and Remark~\ref{rem:conj}, we can see that $W^{ss}(\omega,x )=I(W^{uu}(I(\omega,x);\Psi)$, where $\Psi=\tau\ltimes \mathscr{F}^{-1}$. Therefore, we can simplify the analysis to focus only on the strong unstable leaves of $\Phi$.
\end{obs}

The following proposition provides another way to express the global strong leaves for $\Phi$.
\begin{prop}\label{propo-otra-forma-expresar-variedad}
Given $(\omega,x)\in\Sigma_{k}\times X$, we have
\begin{equation*}
       W^{uu}(\omega,x )
       =\bigcup_{n\geq1}{\bigcup_{\sigma\in\{1,\dots,k\}^{n}}{
       W_{loc}^{u}
       (\omega_{\sigma} )\times\{f_{\sigma}\circ f_{\omega}^{-n}(x)\}
       }}.
\end{equation*}
where $\sigma=\sigma_{1}\dots\sigma_{n}\in\{1,\dots,k\}^{n}$,
 $f_{\sigma}=f_{\sigma_{n}}\circ\dots \circ f_{\sigma_{1}}$ and
$\omega_{\sigma}=\dots\omega_{-(n+1)}\sigma_{1}\dots\sigma_{n};\omega_{0}\omega_{1}\dots$
\end{prop}
\begin{proof}
We have
\begin{align*}
      W^{uu}(\omega,x )
       &=\bigcup_{n\geq0}{\Phi^{n}\left(W_{loc}^{uu}
       (\tau^{-n}(\omega),f_{\omega}^{-n}(x) )\right)}
       =\bigcup_{n\geq0}{\Phi^{n}\left(W_{loc}^{u}(\tau^{-n}
       (\omega) )\times\{f_{\omega}^{-n}(x)\}\right)}\\
       &=\bigcup_{n\geq0}{\Phi^{n}
       \bigg(\big(\bigcup_{\sigma\in\{1,\dots,k\}^{n}}{W_{loc}^{u}
       (\tau^{-n}(\omega) )\cap[\sigma]^{+}}\big)\times
       \{f_{\omega}^{-n}(x)\}\bigg)}  \\ 
       &=\bigcup_{n\geq0}{\bigcup_{\sigma\in\{1,\dots,k\}^{n}}{\Phi^{n}\bigg(\left(W_{loc}^{u}
       (\tau^{-n}(\omega) )\cap[\sigma]^{+}\right)\times\{f_{\omega}^{-n}(x)\}\bigg)}}\\
       &=\bigcup_{n\geq0}{\bigcup_{\sigma\in\{1,\dots,k\}^{n}}W_{loc}^{u}(\omega_{\sigma} )
       \times\{f_{\sigma}\circ f_{\omega}^{-n}(x)\}}.
    \qedhere
\end{align*}
\end{proof}

Consider the projections onto the base and fiber spaces, given by $\Pi_{\Sigma_{k}}:\Sigma_{k}\times X\to \Sigma_{k}$ and $\Pi_{X}:\Sigma_{k}\times X\to X$, respectively. Let $F_\mathscr{F}$ denote the Hutchinson operator associated with $\mathscr{F}$, introduced in~\S\ref{sec:Hutchinson}. With this notation, we can use the previous proposition to obtain the following expressions:
\begin{align}
 \Pi_{\Sigma_{k}}(W^{uu}(\omega,x ))
  &=\bigcup_{n\geq0}{\bigcup_{\sigma\in\{1,\dots,k\}^{n}}{W_{loc}^{u}(\omega_{\sigma} )}}; \notag \\  \label{eq:proj-wu}
 \Pi_{X}(W^{uu}(\omega,x ))&=\bigcup_{n\geq0}
    {\bigcup_{\sigma\in\{1,\dots,k\}^{n}}{\{f_{\sigma}\circ f_{\omega}^{-n}(x)\}}}
    =\bigcup_{n\geq0}{F^{n}_{\mathscr{F}}(f_{\omega}^{-n}(x))}.
\end{align}

\begin{obs} \label{rem:projXWss}
    The dual statement of Proposition~\ref{propo-otra-forma-expresar-variedad} also holds for $W^{ss}(\omega,x )$. Moreover, we have
$$
 \Pi_{X}(W^{ss}(\omega,x ))
    =\bigcup_{n\geq0}{F^{n}_{\mathscr{F}^{-1}}(f_{\omega}^{n}(x))}.
$$
\end{obs}
\begin{obs}\label{rem:projX}
In order to show that the global strong unstable leaf $W^{uu}(\omega,x )$ is dense in $\Sigma_k \times X$, it is necessary to establish that both $\Pi_{\Sigma_{k}}(W^{uu}(\omega,x ))$ and $\Pi_{X}(W^{uu}(\omega,x ))$ are dense in their respective spaces, $\Sigma_k$ and $X$. The same observation holds for~$W^{ss}(\omega,x )$.
\end{obs}

In Lemma 5.5 of \cite{HN13}, the authors state without proof that the strong unstable foliation $\mathcal{F}^{uu}(\Phi)$ is minimal under the assumption that $\mathscr{F}$ is minimal. However, this claim does not hold in general, as demonstrated by the following example.

\begin{exem} \label{example_rotation}
Let $f$ be an irrational rotation of $\mathbb{S}^1$. Consider $\mathscr{F}={f}$ and $\Phi=\tau \ltimes \mathscr{F}=\tau\times f$ on $\Sigma_1\times \mathbb{S}^1$. Clearly, $\mathscr{F}$ is minimal and $F^n_{\mathscr{F}}(f_\omega^{-n}(x))=\{x\}$ for all $(\omega,x)\in \Sigma_1\times X$.
By using~\eqref{eq:proj-wu}, we have that $\Pi_{\mathbb{S}^1}(W^{uu}(\omega,x ))= \{x\}$. According to Remark~\ref{rem:projX}, the global strong unstable leaf $W^{uu}(\omega,x )$ is not dense in $\Sigma_1\times \mathbb{S}^1$.
\end{exem}

As stated in Theorem~\ref{mainthm1}, we will now prove the density of the leaves of the strong unstable foliation $\mathcal{F}^{uu}(\Phi)$ of $\Phi = \tau \ltimes \mathscr{F}$ by satisfying a slightly stronger criterion than the minimality of $\mathscr{F},$ namely, assuming that $X$ is a strict attractor of $\mathscr{F}$.


\begin{proof}[Proof of Theorem~\ref{mainthm1}] Let $k\geq 1$ be the number of elements in $\mathscr{F}$. Take
$(\omega,x)\in\Sigma_{k}\times X$ and consider a basic open set  $D\times U$ in $\Sigma_{k}\times X$. That is, $U$ is an open set of $X$ and $D=[\alpha;\alpha']$ is a cylinder in $\Sigma_{k}$  where $\alpha=\alpha_{1}\alpha_{2}\dots\alpha_{r}$ and $\alpha'=\alpha'_{0}\alpha'_{1}\dots\alpha'_{r'}$
are finite words in the alphabet~$\{1,\dots,k\}$.

\begin{afir}\label{afi-previo-minimalidad}
There is  $n\in\mathbb{N}$ with $n>r$ and a finite word
$\beta=\beta_{1}\beta_{2}\dots\beta_{m}$ in the alphabet $\{1,\dots,k\}$ with $m<n$ such that
\begin{enumerate}[label=(\roman*)]
    \item\label{item1-afi-previo-minimalidad} $m+r=n$;
    \item\label{item2-afi-previo-minimalidad} $f_{\alpha}\circ f_{\beta}\circ f_{\omega}^{-n}(x)\in U$ where
    $f_{\alpha}=f_{\alpha_{r}}\circ\dots\circ f_{\alpha_{1}}$ and $f_{\beta}=f_{\beta_{m}}\circ\dots\circ f_{\beta_{1}}$;
    \item\label{item3-afi-previo-minimalidad}$W_{loc}^{u}(\omega_{\beta\alpha} )\cap D\neq \emptyset$ where
    $\omega_{\beta\alpha}=\dots\omega_{-(n+1)}\beta_{1}\beta_{2}\dots\beta_{m}\alpha_{1}\alpha_{2}\dots\alpha_{r};\omega_{0}\omega_{1}\dots$
\end{enumerate}
\end{afir}
\begin{proof}
Notice that the Hutchinson operator $F=F_\mathscr{F}$ associated with $\mathscr{F}$ is continuous and clearly a monotone map. Then, since $X$ is a strict attractor of $\mathscr{F}$ (i.e., of  $F$) by Proposition~\ref{propo-previo-minimalidad},
setting $\varepsilon>0$ such that  $(f_{\alpha})^{-1}(U)$ contains a ball of radius $\varepsilon$, there is
$m\in\mathbb{N}$ such that
\begin{equation*}
    F^{N}(\{y\})\cap (f_{\alpha})^{-1}(U)\neq\emptyset,\quad\text{for all}\quad N\geq m\quad\text{and for every}\quad y\in X.
\end{equation*}
In particular,
\begin{equation}\label{iv.4}
    F^{m}(\{f_{\omega}^{-(m+r)}(x)\})\cap (f_{\alpha})^{-1}(U)\neq\emptyset.
\end{equation}
Let $n=m+r$. From~\eqref{iv.4}, we have that there is a finite word $\beta=\beta_{1}\beta_{2}\dots\beta_{m}$  in the alphabet $\{1,\dots, k\}$ such that
$f_{\beta}(f_{\omega}^{-n}(x))\in (f_{\alpha})^{-1}(U)$.
  Then $f_{\alpha}\circ f_{\beta}\circ f_{\omega}^{-n}(x)\in U$ which proves~\ref{item1-afi-previo-minimalidad} and~\ref{item2-afi-previo-minimalidad}. Trivially,~\ref{item3-afi-previo-minimalidad} follows.
\end{proof}


With the notation of Claim~\ref{afi-previo-minimalidad}, let us consider the finite word $\sigma=\beta\alpha \in \{1,\dots,k\}^n$ formed by the concatenation of the words $\beta$ and $\alpha$.
We have $f_{\sigma}=f_{\alpha}\circ f_{\beta}$ and $\omega_{\beta\alpha}=\omega_\sigma$. Using item~\ref{item2-afi-previo-minimalidad} and~\ref{item3-afi-previo-minimalidad} from the claim, we obtain that $(W_{loc}^{u}(\omega_{\sigma} )\times {f_{\sigma}\circ f_{\omega}^{-n}(x)})\cap(D\times U)\neq\emptyset$, which implies, by Proposition~\ref{propo-otra-forma-expresar-variedad}, that $W^{uu}(\omega,x )\cap (D\times U)\neq\emptyset$. Therefore, we conclude that the strong unstable leaf $W^{uu}(\omega,x )$ is dense in $X$, and hence the minimality of the foliation is proved.
\end{proof}



The counterexample presented in Example~\ref{example_rotation} of~\cite[Lemma 5.5]{HN13} is based on a family $\mathscr{F}$ consisting of only one element. It is natural to ask whether the strong unstable foliation $\mathcal{F}^{uu}(\Phi)$ can be minimal when $\mathscr{F}$ is minimal and the semigroup $\langle \mathscr{F}\rangle^+$ is free with at least two generators. However, the following corollary of Theorem~\ref{mainthm1} shows that the answer is still negative. This result provides also a non-trivial example in the study of IFS, highlighting the difference between strict attractors and minimality.

\begin{cor}\label{exemplo-pablo}
Let $f_{1}$ and $f_{2}$ be rotation circle diffeomorphisms of angle, respectively,  $\alpha,\beta\in\mathbb{R}\setminus\mathbb{Q}$
with
$\beta-\alpha\in\mathbb{Q}$. Set $\mathscr{F}=\{f_1,f_2\}$ and $\Phi=\tau\ltimes \mathscr{F}$ from $\Sigma_2\times \mathbb{S}^1$ to itseft.
Then
\begin{enumerate}
\item $\mathscr{F}$ is minimal,
     \item $\mathbb{S}^{1}$ is not a strict attractor for $\mathscr{F}$ and
     \item  the strong unstable foliation $\mathcal{F}^{uu}(\Phi)$ is not minimal.
\end{enumerate}
\end{cor}
\begin{proof}
Let $\omega=111\dots$ be the infinite sequence whose digits are all one. Given $n\in\mathbb{N}$ and $\sigma\in\{1,2\}^{n}$, consider $0\leq n_{\sigma}\leq n$
the number of times that the digit $2$ appears in the word $\sigma$. Thus, for any $x\in\mathbb{S}^{1}$  we have
\begin{equation*}
    \begin{aligned}
        f_{\sigma}(f_{\omega}^{-n}(x))=& f_{\sigma}(x-n\alpha)
        = x-n\alpha+n_{\sigma}\beta+(n-n_{\sigma})
        \alpha
        =x+n_{\sigma}(\beta-\alpha)
        =R_{\beta-\alpha}^{n_{\sigma}}(x).
    \end{aligned}
\end{equation*}
Consider $F$ the Hutchinson operator associated with $\mathscr{F}$. Note that
\begin{equation} \label{eq:orb-rotation}
   \bigcup_{n\in\mathbb{N}}F^{n}(f_{\omega}^{-n}(x))=\bigcup_{n\in\mathbb{N}} \bigcup_{\sigma\in\{1,2\}^{n}}{f_{\sigma}(f_{\omega}^{-n}(x))}
   =\bigcup_{n\in\mathbb{N}}\bigcup_{\sigma\in\{1,2\}^{n}}{R_{\beta-\alpha}^{n_{\sigma}}(x)} \subset \mathcal{O}^+_{R_{\beta-\alpha}}(x).
\end{equation}
 Since $R_{\beta-\alpha}$ is a rational rotation, the orbit $\mathcal{O}^+_{R_{\beta-\alpha}}(x)$ of $x$ is a finite set. Then, in view of~\eqref{eq:orb-rotation} and~\eqref{eq:proj-wu}, we get that $\Pi_{\mathbb{S}^1}(W^{uu}(\omega,x ))$ is finite where $\Phi=\tau\ltimes \mathscr{F}$. However, if $\mathbb{S}^1$ will be a strict attractor of $\mathscr{F}$, by Theorem~\ref{mainthm1}, this set should be dense in $\mathbb{S}^1$ and thus, infinite.
 This completes the proof of the proposition.
\end{proof}

\subsection{Genericity of minimal strong foliations}
Recall that two maps $f$ and $g$ are conjugated if there exists a homeomorphism $h$ such that $h \circ f = g \circ h$.

\begin{prop} \label{eq:conj}
Let $\mathscr{F}$ and $\mathscr{G}$ be two families that are identified by reordering their elements. Then,
$\Phi=\tau\ltimes \mathscr{F}$ and $\Psi=\tau\ltimes \mathscr{G}$ are conjugated.
\end{prop}

\begin{proof}
Let $\sigma$ be the permutation of the elements of $\mathscr{G}$ that provides $\mathscr{F}$.
Consider the complement map $h:\Sigma_k \to \Sigma_k$, which changes the label of each symbol according to the permutation~$\sigma$. Then, the map $H=h\times \mathrm{id}$ satisfies $H\circ \Phi = \Psi \circ H$.
\end{proof}


Let $1\leq r\leq \infty$ and $k\geq 1$ be fixed. Recall the definitions of the metric spaces  $\IFS^r_k(\mathbb{S}^1)$, $\mathrm{RT}_k^r(\mathbb{S}^1)$, $\mathrm{S}_{k}^{r}(\mathbb{S}^{1})$ and $\mathrm{SRT}_k^r(\mathbb{S}^1)$   introduced in~\S\ref{sec:gen-min-intro} and~\S\ref{IFS-intro}.  By Proposition \ref{eq:conj} and the fact that the families in  $\IFS_k^r(\mathbb{S}^1)$
are considered equal regardless of the order of their
element, we get that the symbolic one-step skew-product  maps in $\mathrm{S}_{k}^{r}(\mathbb{S}^{1})$ are identified up to conjugation. Moreover,
in view of Theorem~\ref{teo:skew-transitive}, we have that
\begin{equation*}
    \mathrm{SRT}_{k}^{r}(\mathbb{S}^{1})=\{\Phi=\tau\ltimes \mathscr{F}\in\mathrm{S}_{k}^{r}(\mathbb{S}^{1}): \, \mathscr{F}\in \mathrm{RT}_k^r(\mathbb{S}^1)\}.
\end{equation*}



\begin{proof}[Proof of Theorem~\ref{mainthm2}]
Let $\mathcal{W}$ be the open and dense subset of $\mathrm{RT}_{k}^{r}(\mathbb{S}^{1})$ obtained in Theorem~\ref{mainthm-IFS}. Consider
$\mathcal{R}\eqdef\{ \Phi=\tau\ltimes\mathscr{F}\in\mathrm{SRT}_{k}^{r}(\mathbb{S}^{1}): \mathscr{F}\in\mathcal{W}\}$.
It is clear that $\mathcal{R}$ is both open and dense in $\mathrm{SRT}_{k}^{r}(\mathbb{S}^{1})$. Moreover, for any $\Phi=\tau\ltimes\mathscr{F}\in \mathcal{R}$, we have that $\mathscr{F} \in \mathcal{W}$, which implies that $\mathbb{S}^1$ is a strict attractor for both $\mathscr{F}$ and $\mathscr{F}^{-1}$. Therefore, by applying Theorem~\ref{mainthm1} and Remark~\ref{obs:duality}, we conclude that both the strong stable and strong unstable foliations of $\Phi$ 
are minimal.

On the other hand, if $\Psi=\tau\ltimes\mathscr{G}\in \mathrm{S}^r_k(\mathbb{S}^1)$ is transitive, it follows by Theorem~\ref{teo:skew-transitive} that $\mathscr{G}$ is transitive. Therefore, Theorem~\ref{mainthm-IFS} implies that $\mathscr{G}$ can be approximated by IFSs in $\mathcal{W}$. In other words, $\Psi$ can be approximated by a skew-product in $\mathcal{R}$ as required.
\end{proof}

Finally, we complete the proof of Theorem~\ref{mainthm-examples}.

\begin{proof}[Proof of Theorem~\ref{mainthm-examples}]
Let $\mathscr{F}$ and $\mathscr{G}$ be the families in $\mathrm{RT}^r_{k+1}(\mathbb{S}^1)$ and $\mathrm{RT}^r_{k}(\mathbb{S}^1)$, respectively, as described in Theorem~\ref{mainteo:exemplos-IFS} with $k\geq 8$ and $1\leq r\leq \infty$. By Proposition~\ref{prop:forward-not-backward} and Corollary~\ref{cor:rt}, both families contain the symmetric family $\mathscr{G}_0$ described in Proposition~\ref{propo-existencia-minimal-excep}. Moreover, the group $G=\langle \mathscr{G}_0\rangle^+$ has an exceptional minimal set $K$, and $f(K)\supset K$ for all $f\in \mathscr{F} \cup \mathscr{G}$.

Consider $\Phi=\tau\ltimes \mathscr{F}$ and $\Psi=\tau\ltimes \mathscr{G}$, which belong to $\mathrm{SRT}^r_{k+1}(\mathbb{S}^1)$ and $\mathrm{SRT}^r_k(\mathbb{S}^1)$, respectively. By Theorem~\ref{mainthm1}, the strong unstable foliation $\mathcal{F}^{uu}(\Phi)$ is minimal because $\mathbb{S}^1$ is a strict attractor of $\mathscr{F}$ as established in Proposition~\ref{prop:forward-not-backward}.

To show that $\mathcal{F}^{ss}(\Phi)$ is not minimal, we take $x\in K$ and a sequence $\omega\in \Sigma_k$ such that $f_{\omega}^n \in G$ for all $n\geq 1$. Since $f(K)\subset K$ for all $f\in \mathscr{F}^{-1}$, we have $F_{\mathscr{F}^{-1}}^n(f^n_\omega(x))\subset K$ for all $n\geq 0$, where $F_{\mathscr{F}^{-1}}$ is the Hutchinson operator associated with $\mathscr{F}^{-1}$. Hence, by Remarks~\ref{rem:projXWss} and~\ref{rem:projX}, we have $\Pi_X(W^{ss}(\omega,x )) \subset K$ and is not dense. Therefore, $\mathcal{F}^{ss}(\Phi)$ is not minimal.

The same argument can be applied to show that neither $\mathcal{F}^{ss}(\Psi)$ nor $\mathcal{F}^{uu}(\Psi)$ is minimal, completing the proof of the theorem.
\end{proof}


\section*{Acknowledgement}
The results presented in this paper form part of the doctoral thesis of the second author, J.~A.~Cisneros.
The first author, P.~G.~Barrientos, would like to express his deepest gratitude for the support provided by grant PID2020-113052GB-I00 funded by MCIN, PQ 305352/2020-2 (CNPq), and JCNE E-26/201.305/2022 (FAPERJ).

\end{document}